\documentclass[review]{siamart190516}

\usepackage{psfrag}
\usepackage{epsfig}
\usepackage{amsmath}
\usepackage{amssymb}
\usepackage{amsfonts}
\usepackage{multirow}
\usepackage{subfig}
\usepackage{algpseudocode} 
\usepackage{color}
\usepackage{bbold}
\usepackage{tabularx} 
\usepackage{booktabs} 

\usepackage{enumitem}  

\newtheorem{remark}{Remark}[section]

\renewcommand{\thefootnote}{\arabic{footnote}}
\DeclareMathOperator{\tr}{tr}
\DeclareMathOperator*{\argmin}{argmin}

\title{Parallel Energy-Minimization Prolongation for Algebraic Multigrid}
\author{ Carlo Janna\footnotemark[1] \footnotemark[2]
    \and Andrea Franceschini\footnotemark[3] \footnotemark[4]
    \and Jacob B. Schroder\footnotemark[5]
    \and Luke Olson\footnotemark[6] }

\begin{document}
\nolinenumbers
\maketitle

\renewcommand{\thefootnote}{\fnsymbol{footnote}}
\footnotetext[1] {M$^3$E S.r.l., via Giambellino 7, 35129 Padova, Italy, {\tt e-mail}
c.janna@m3eweb.it}
\footnotetext[2] {Department ICEA, University of Padova, via Marzolo 9, 35131 Padova,
Italy}
\footnotetext[3]{corresponding author}
\footnotetext[4] {Department ICEA, University of Padova, via Marzolo 9, 35131 Padova,
Italy, {\tt e-mail} andrea.franceschini@unipd.it}
\footnotetext[5] {Department of Mathematics and Statistics, University of New Mexico,
Albuquerque, NM 87131, USA, {\tt e-mail} jbschroder@unm.edu}
\footnotetext[6] {Siebel Center for Computer Science, University of Illinois at
Urbana-Champaign, 201 N. Goodwin Ave., Urbana, IL 61801, USA, {\tt e-mail}
lukeo@illinois.edu}
\renewcommand{\thefootnote}{\arabic{footnote}}

\begin{abstract}
Algebraic multigrid (AMG) is one of the most widely used solution techniques for linear systems
of equations arising from discretized partial differential equations. The popularity of AMG stems from its potential to
solve linear systems in almost linear time, that is with an 
$O(n)$ complexity, where $n$ is the problem size. This capability is crucial at the 
present, where the increasing availability of massive HPC platforms pushes for
the solution of very large problems. The key for a rapidly converging AMG method is a good
interplay between the smoother and the coarse-grid correction, which in turn requires 
the use of an effective prolongation. From a theoretical viewpoint,
the prolongation must accurately represent near kernel components and, at the same time,
be bounded in the energy norm.
For challenging problems, however, ensuring both these requirements is not easy and
is exactly the goal of this work. We propose a constrained minimization procedure
aimed at reducing prolongation energy while preserving the near kernel components in the span of interpolation.
The proposed algorithm is based on previous energy minimization 
approaches utilizing a preconditioned restricted conjugate gradients method, 
but has new features and a specific focus on parallel performance and implementation.  
It is shown that the resulting solver, when used for large real-world problems from various
application fields, exhibits excellent convergence rates and scalability and outperforms
at least some more traditional AMG approaches.
\end{abstract}


\newcommand{\I}{\mathcal{I}}
\newcommand{\J}{\mathcal{J}}
\newcommand{\LL}{\mathcal{L}}

\newcommand{\rt}[1]{\textcolor{red}{#1}}
\newcommand{\bt}[1]{\textcolor{blue}{#1}}
\newcommand{\gt}[1]{\textcolor{green}{#1}}
\newcommand{\rtb}[1]{\textbf{\textcolor{red}{#1}}}
\newcommand{\teal}[1]{\textbf{\textcolor{teal}{#1}}}

\newcommand{\ww}{\widehat{w}}
\newcommand{\diag}{\mathrm{diag}}
\newcommand{\low}{\mathrm{low}}
\newcommand{\lmax}{l_{\scriptsize \mbox{max}}}

\newcommand{\rowsubi}[1]{\noindent #1_{i}}
\newcommand{\row}[1]{\noindent \mathbb{#1}}
\newcommand{\rowtent}[1]{\noindent \mathbb{#1}_0}
\newcommand{\localB}{\mathbb{B}}  
\newcommand{\localQ}{Q}  

\section{Introduction}

With the increasing availability of powerful computational resources, scientific and engineering
applications are becoming more demanding in terms of both memory and CPU time. For common
methods used in the numerical approximation to partial differential equations (e.g., finite difference, finite volume, or finite element),
the resulting approximation can easily grow to several millions or even billions of unknowns.
The efficient solution to the associated sparse linear system of equations
\begin{equation}\label{system}
A \mathbf{x} = \mathbf{b},
\end{equation}
either as a stand-alone system or as part of a nonlinear solve process, often
represents a significant computational expense in the numerical application.  Thus, research on sparse linear solvers
continues to be a key topic for efficient simulation at large scales.  One of the most popular sparse linear
solvers is algebraic multigrid (AMG)~\cite{BrMcRu1982,BrMcRu1984,RuStu1987} because of its potential for $O(n)$ computational cost
in number of degrees-of-freedom $n$ for many problem types.

A fast converging AMG method relies on the complementary action of
relaxation (e.g., with weighted Jacobi) and coarse grid correction, which is a
projection step focused on eliminating the error that is not reduced by
relaxation.  Even in a purely algebraic setting, the main algorithmic decisions
in multigrid are often based on heuristics for elliptic problems.  As a result, for more complex
applications, traditional methods often break down, requiring additional techniques to improve
accuracy with a careful eye on overall computational complexity.

Even with advanced AMG methods, robustness remains an open problem for a variety of
applications, especially in parallel.  Yet, there have been several advances in
recent years that have significantly improved convergence in a range of
settings. Adaptive AMG~\cite{maclachlan2006adaptive} and adaptive smoothed
aggregation~\cite{BreFalMacManMccRug05} are among early attempts to assess the quality of the
AMG setup phase \textit{during} the setup process, with the ability to adaptively improve 
the interpolation operators.  Later works focus on extending the adaptive ideas
to more general settings~\cite{paludetto2019novel}, and in particular,
Bootstrap AMG~\cite{bootstrap2011} further develops the idea of adaptive interpolation with least-squares interpolation coupled with locally relaxed vectors and multilevel eigenmodes.
Other advanced approaches have a focus on specific AMG components, such as energy minimization
of the interpolation operator~\cite{ManBreVan99,WaChSm2000,SalTum08,OlsSchTum11,ManOlsSchSou17},
generalizing the strength of
connection procedure~\cite{olson2010new,BrBrKaLi2015}, or by considering the nonsymmetric nature of the problem
directly~\cite{manteuffel2018nonsymmetric,manteuffel2019nonsymmetric}.

While AMG robustness and overall convergence has improved with combinations of the advances above,
the overarching challenge of controlling cost is persistent.  In this paper, we make
a number of related contributions with a focus on AMG effectiveness and efficiency at large scale.
Our key contributions are as follows:
\begin{itemize}

   \item The quality and sparsity of tentative interpolation is improved through a
      novel utilization of sparse $QR$ and a new process for sparsity pattern
      expansion that targets locally full-rank matrices for improved mode
      interpolation constraints;

   \item We accompany the energy minimization construction of interpolation
      with new energy and convergence monitoring, thus limiting the total cost;

   \item We apply a new preconditioning technique for the energy
      minimization process based on Gauss-Seidel
      applied to the blocks;

   \item We present the non-trivial and efficient parallel implementation in detail; and

   \item We demonstrate improved convergence and computational complexity with several large scale experiments.


\end{itemize}
The remainder of the paper is as follows.
We begin with the basics of AMG in~\Cref{sec:amg}.
In~\Cref{sec:emin}, we derive the energy minimization process based on QR
factorizations and introduce a method for monitoring reduction of energy in
practice. Finally, we conclude with several numerical experiments
in~\Cref{sec:numerics} along with a discussion on performance.


\section{Introduction to Classical AMG}\label{sec:amg}


The effectiveness of AMG as a solver depends on the complementary relationship
between relaxation and coarse-grid correction, where the error not reduced by
relaxation on the fine grid (e.g., with weighted-Jacobi or Gauss-Seidel) is accurately represented on
the coarse grid, where a complementary error correction is computed. For a more
in-depth introduction to AMG, see the works \cite{BrHeMc2000,TrOo2001}.  Here, we focus
our description of AMG on the coarse grid and interpolation setup, which are most relevant to the rest of the paper.

Constructing the AMG coarse grid begins with a partition of the $n$ unknowns of $A$
into a C-F partition of
$n_f$ fine nodes and $n_c$ coarse nodes: $\{0,\ldots,n-1\} = \mathcal{C}\cup\mathcal{F}$.
From this, we assume an ordering of $A$ by F-points followed by C-points:
\begin{equation}\label{CFspl}
A = \left[\begin{array}{cc}
A_{ff}   & A_{fc} \\
A_{fc}^T & A_{cc} \\
\end{array}\right],
\end{equation}
where for example, $A_{ff}$ corresponds to entries in $A$ between two F-points.
We also assume $A$ is SPD so that $A_{cf} = A_{fc}^T$. In classical
AMG, prolongation takes the form
\begin{equation}\label{prol}
P = \left[\begin{array}{c}
W \\
I
\end{array}\right],
\end{equation}
where $W$ must be sparse (for efficiency) and represents interpolation from the coarse grid to fine grid $F$-points.

In constructing prolongation of the form~\cref{prol}, there are two widely accepted guidelines, the so-called  \textit{ideal}~\cite{BraFal10,xu2018ideal} and
\textit{optimal}~\cite{XuZik17,brannick2018optimal} forms of prolongation. Although both of these
are not feasible in practical applications, leading to very expensive and dense
prolongation operators, the concepts behind their definition are valuable guides
for constructing effective $P$.

Ideal prolongation is constructed by starting with the above C-F partition and
constructing $P_{\text{id}}$ as
\begin{equation}
P_{\text{id}} =
\left[
  \begin{array}{c}
-A_{ff}^{-1} A_{fc}\\
I \\
\end{array}
\right].
\end{equation}
Making $P_{\text{id}}$ the goal for interpolation is motivated by Corollary
3.4 from the theoretical work~\cite{FaVa2004}. Here, the main assumption is a
classical AMG framework where $P$ is of the form in equation (\ref{prol}).\footnote{The other
assumptions are specific choices for the map to F-points $S = [I, 0]$, for the map to
C-points $R = [0, I]$, and for relaxation $X = \| A \| I$.}  In this
setting, the choice of $W = -A_{ff}^{-1} A_{fc}$ minimizes
the two-grid convergence of AMG relative to the choice of $P$, i.e., relaxation is fixed.
Motivating our later energy minimization approach,
$P_{\text{id}}$ can be viewed as having zero energy rows, as $A P_{\text{id}}$
is zero at all F-rows.  Additionally, this classical AMG perspective likely
makes the task of energy minimization easier, in that the conditioning of
$A_{FF}$ is usually superior to that of $A$.

With \textit{optimal} interpolation, the goal of interpolation is to capture
the algebraically smoothest modes in span$(P)$, i.e., the modes left behind by relaxation.
More specifically following~\cite{brannick2018optimal}, let
$\lambda_1 \le \lambda_2 \le \dots \le \lambda_n$ and $v_1, v_2, \dots, v_n$ be
the ordered eigenvalues and eigenvectors of the generalized eigenvalue
problem $A x = \lambda \tilde{M} x$, where $\tilde{M}$ is the symmetrized
relaxation matrix, e.g., the diagonal of $A$ for Jacobi.
(See the work~\cite{brannick2018optimal} for more details.)  Then, the two-grid convergence of AMG
is minimized if
\begin{equation}
   \label{popt}
   \mbox{span}(P) = \mbox{range}( v_1, v_2, ..., v_{n_c} ).
\end{equation}
Note, that no assumptions on the structure of $P$ are made, as in equation (\ref{prol}).
Motivating our later energy minimization approach, equation (\ref{popt}) indicates
that span$(P)$ should capture low-energy modes relative to relaxation,
which our Jacobi or Gauss-Seidel preconditioned energy minimization approach
will explicitly target.  Moreover, our energy minimization approach will incorporate
constraints which explicitly force certain vectors to be in span$(P)$,
where these vectors are chosen to represent the $v_i$ with smallest eigenvalues.

The idea of energy minimization AMG with constraints has been exploited for both symmetric and
non-symmetric operators in several works~\cite{ManBreVan99,WaChSm2000,SalTum08,OlsSchTum11,ManOlsSchSou17},
and, though requiring more computational effort than classical interpolation
formulas, often provides improved preconditioners that balance the extra
cost.

\section{Energy minimization prolongation}\label{sec:emin}

The energy minimization process combines
the key aspects of ideal and optimal prolongation.  To define this, we first introduce
$V$, a basis for the \textit{near kernel} of $A$ or the \textit{lowest energy} modes of $A$.
Then, energy minimization seeks
to satisfy two requirements:
\begin{enumerate}[font=\bfseries]
  \item \textbf{Range:} The range of prolongation must include the near kernel $V$:
\begin{equation}\label{inKern}
V \subseteq \mathrm{range}(P)
\end{equation}
\item \textbf{Minimal:} The energy of each column of $P$ is minimized:
\begin{equation}\label{trace}
  P = \argmin_{P} \left(\mbox{tr}(P^T A P)\right).
\end{equation}
\end{enumerate}
To construct a $P$ that contains $V$ in the range and has minimal energy, we next introduce the key components
needed by most energy minimization approaches, namely
\begin{enumerate}[label=\roman*)]  
  \item a sparsity pattern for efficient application of $P$;
  \item a constraint to enforce~\cref{inKern}; and
  \item an approximate block diagonal linear system for solving equation~\cref{trace}.
\end{enumerate}
For practical use in AMG, the prolongation operator must be sparse,
therefore construction
begins by defining a sparse non-zero pattern for $P$. Assume that a
strength of connection (SoC) matrix $S$ is provided where nonzero entries denote a strong coupling between degrees-of-freedom 
\cite{RuStu1987,olson2010new,BrBrKaLi2015}.
Next, let $P_0$ be a \textit{tentative} prolongation with 
non-zero pattern $\mathcal{P}_0$, to be used as an initial guess for $P$.\footnote{See
\cref{sec:construct-ptent} for our contributions regarding the construction of $P_0$,
which is based on the adaptive algorithm~\cite{IsoFriSpiJan21}. $P_0$ can be
defined similarly to the tentative prolongation from smoothed aggregation AMG~\cite{Van96}
in that $P_0$ interpolates the basis $V$, but needs further improvement, for example,
with energy minimization.}
We next obtain an enlarged sparsity pattern $\mathcal{P}$ by growing 
$\mathcal{P}_0$ to include all strongly connected
neighbors up to distance $k$. 
Denoting with $\overline{P}$ the unitary matrix obtained from $P$ by replacing its
non-zeros with unitary entries, this is equivalent to 
\begin{equation}\label{enlPatt}
   \overline{P}_k = S^k \overline{P}_0,
\end{equation}
where $\mathcal{P}$ is the pattern of $\overline{P}_k$ (see~\cite{OlsSchTum11}).
%

For a constraint condition satisfying~\cref{inKern}, we
start by splitting the near kernel basis $V$ with the same C-F
splitting as in~\cref{CFspl}:
\begin{equation}
V = \left[\begin{array}{c}
V_f \\
V_c \\
\end{array}\right].
\end{equation}
Recalling the form of interpolation~\cref{prol}, the near kernel 
requirement for $P$ becomes
\begin{equation}\label{TSP}
W \; V_c = V_f,
\end{equation}
which is a set of $n_f$ conditions on the rows of $W$.
By denoting $\rowsubi{w}^T$ as the $i$-th row of $W$ (and $P$) and $\rowsubi{v}^T$ as the $i$-th row $V$,
condition~\cref{TSP} is then exploited row-wise as
\begin{equation}\label{TSP_cond}
   V_c^T \rowsubi{w} = \rowsubi{v} \qquad \forall i \in \mathcal{F}.
\end{equation}
Using the sparsity pattern $\mathcal{P}$, we rewrite \cref{TSP_cond} for only the
nonzeros in each row $\rowsubi{w}$.  Letting the index set $\mathcal{J}_i$ be
the column indices in the $i$-th row of $P$, this becomes
\begin{equation}\label{Constr}
   V_c(\mathcal{J}_i,:)^T \overline{\rowsubi{w}} = \rowsubi{v} \qquad \forall i \in {\mathcal F},
\end{equation}
where $\overline{\rowsubi{w}} = \rowsubi{w} (\mathcal{J}_i)$ collects only the nonzeros of $\rowsubi{w}$. 
It is
important to note that for each of the $n_f$ fine points, the constraints~\cref{Constr},
are independent of each other, because each
entry of $W$ appears in only one constraint. Denote by $\widetilde{p}$
the column vector collecting the nonzero entries of $P$ \emph{row-wise}.
Similarly, we denote by $p$ the column vector containing the nonzero entries of $P$ \emph{column-wise}. By definition, $\widetilde{p}$ and $p$ have the same size, equal to the number of nonzeros in $P$.
Next, we write
the constraints in the following matrix form:
\begin{equation}\label{Constr_Mat}
\widetilde{B}^T \widetilde{p} = \widetilde{g},
\end{equation}
where $\widetilde{g}$ collects the $\rowsubi{v}$ in~\cref{Constr} into a single vector,
and where $\widetilde{B}^T$, is a block diagonal matrix
composed of $V_c(\mathcal{J}_i,:)^T$ due to the independence of constraints.

Lastly, we describe the reduced linear system framework for approximating~\cref{trace}.
Minimizing~\cref{trace} is equivalent to minimizing the energy of the individual
columns of $P$ on the prescribed nonzero pattern:
\begin{equation}\label{mincol}
p_i = \argmin_{p_i \in \mathcal{P}_i} p_i^T A p_i,
\end{equation}
where $p_i$ is the $i$-th column of $P$ and $\mathcal{P}_i$ is the sparsity pattern of column $i$.

Next, let $\mathcal{I}_i$ be the set of nonzero row indices of the $i$-th column of $W$ and
let $\overline{h}_i$ be the vector collecting the nonzero entries of the $i$-th column
of $W$.  Then the minimization in~\cref{mincol} defines $\overline{h}_i$ with
\begin{equation}\label{minblk}
A(\mathcal{I}_i,\mathcal{I}_i) \overline{h}_i = - A(\mathcal{I}_i,i) \qquad
\forall i \in \mathcal{C},
\end{equation}
where $A(\mathcal{I}_i,\mathcal{I}_i)$ is a square, relatively dense, submatrix of $A$
corresponding to the allowed nonzero indices $\mathcal{I}_i$ and $A(\mathcal{I}_i,i)$
is a vector corresponding to the $i$-th column of $A$ at the allowed nonzero indices.
Also in this case, each column of $P$ satisfies~\cref{mincol} independently. Thus, denoting by $p$ the
column vector collecting the nonzero entries of $P$ \emph{column-wise} (i.e., a rearranged version of $\widetilde{p}$) 
and denoting by $f$ the vector collecting each $-A(\mathcal{I}_i,i)$, 
the minimization~\cref{mincol} is recast as
\begin{equation}\label{11blk}
K p = f,
\end{equation}
where $K$ is block diagonal with $A(\mathcal{I}_i,\mathcal{I}_i)$ on the $i$-th block.

The two conditions, one on the range of $P$ and one on the minimality of $P$,
together form a constrained minimization problem, whose
solution is the desired energy minimal prolongation.
Casting this problem using Lagrange multipliers results in the 
saddle point system
\begin{equation}\label{sadsys}
\left[ \begin{array}{cc}
K   & B \\
B^T & 0 \\
\end{array} \right] \left[
\begin{array}{c}
p \\
\lambda \\
\end{array}
\right] = \left[
\begin{array}{c}
f \\
g \\
\end{array}
\right].
\end{equation}

The elements in~\cref{sadsys} are the same as those defined in
\cref{Constr_Mat,11blk}, with the exception that $\widetilde{B}$,
$\widetilde{g}$, and $\widetilde{p}$ are reordered following the columns of $P$,
and $\lambda$ is the vector of Lagrange multipliers whose values are not needed
for the purpose of setting up the prolongation.
We emphasize that, with the entries of $p$ enumerated column-wise with respect to $P$, $K$
is block diagonal. Likewise, if $p$ is enumerated following the rows of $P$, then $B$ becomes
block diagonal. Unfortunately, there is no sorting of $p$ able to make both $K$ and $B$
block diagonal at the same time. Nevertheless, it is possible to take advantage of this
underlying structure in the numerical implementation, as will be shown later. 
Leveraging the block structure of $B$ is also important, because, as we will see
in Section~\ref{sect:min}, our algorithm to minimize energy requires several applications of
the orthogonal projector given as
\begin{equation}
\Pi_B = I - B (B^T B)^{-1} B^T.
\end{equation}
%

The system~\cref{sadsys} follows closely the method from ~\cite{OlsSchTum11}.  In sections 
\ref{sec:construct-ptent}--\ref{sec:precon}, we outline our proposed improvements to energy minimization.

\subsection{Minimization through Krylov subspace methods}\label{sect:min}

Following~\cite{OlsSchTum11}, energy minimization proceeds by
starting with a tentative prolongation, $P_0$, that satisfies the near kernel
constraints~(see~\cref{TSP}). Denoting by $p_0$ the tentative prolongation in vector form with nonzero entries
collected \emph{column-wise},
where we highlight that the subscript 0 does not refer to a specific $P$ column as $p_i$ in Eq. \eqref{mincol},
these constraints read
\begin{equation}\label{eq:vec_tent_constraint}
B^T p_0 = g.
\end{equation}
Defining the final prolongation as the tentative $p_0$ plus a correction
$\delta p$ gives
\begin{equation}
p = p_0 + \delta p.
\end{equation}
Then, the problem is recast as finding the optimal correction $\Delta P^*$:
\begin{equation}
\Delta P^* = \argmin_{\Delta P \in \mathcal{P}} \left(\tr((P_0+\Delta P)^T A (P_0+\Delta P))\right),
\end{equation}
subject to the constraint $B^T \delta p = 0$, where $\delta p$ is the vector form of $\Delta P$ with nonzero entries again
collected \emph{column-wise}. By recalling that $\Delta P$ has non-zero
components only in $W$~---~i.e., $\Delta P = [\Delta W^T, 0]^T$~---~and using the C-F
partition~\eqref{CFspl}, we write
\begin{equation}\label{eq:trInc}
\begin{split}
  \tr ((P_0+\Delta P)^T &A (P_0+\Delta P))=\\
&= \underbrace{\tr(\Delta W^T A_{ff} \Delta W)
          + 2 \tr(W_0^T A_{ff} \Delta W)
          + 2 \tr(A_{fc}^T \Delta W)}_{\text{dependent on $\Delta P$}}\\
&+ \underbrace{\tr(W_0^T A_{ff} W_0) + 2 \tr(W_0^T A_{fc})
          +   \tr(A_{cc})}_{\text{independent of $\Delta P$}}.
\end{split}
\end{equation}
Using the preset non-zero pattern for $\Delta W$, the problem is rewritten in vector from to 
minimizie \emph{only} the terms depending on $\Delta P$:
\begin{equation}
   \label{eqn:mindeltaw}
   \argmin_{\delta w} ( \delta w^T K \delta w + 2 w_0^T K \delta w + 2 f^T \delta w),
\end{equation}
subject to the constraint
\begin{equation}\label{constraint}
B^T \delta w = 0,
\end{equation}
where $K$ and $f$ are defined as in~\cref{sadsys} and $\delta w$ and $w_0$ are the vector
forms of $\Delta W$ and $W_0$, respectively, with the nonzero entries collected \emph{column-wise}.

This minimization can be performed using (preconditioned) conjugate gradients 
by ensuring that both the initial solution and the search
direction satisfy the constraint~\cite{OlsSchTum11}. To do this, return to the orthogonal projector
\begin{equation}
\Pi_B = I - B (B^T B)^{-1} B^T,
\end{equation}
and apply conjugate gradients to the singular system
\begin{equation}\label{eq:projSys}
\Pi_B K \Pi_B \delta w = - \Pi_B (f + K w_0),
\end{equation}
starting from $\delta w = 0$. Due to its block diagonal structure, it is straightforward to find a QR
decomposition of $B = QR$, and the projection simply becomes:
\begin{equation}
\Pi_B = I -Q Q^T.
\end{equation}

Finally, introducing $K_\Pi = \Pi_B K \Pi_B$ and $\bar{f} = f + K w_0$, the Krylov subspace built
by conjugate gradients is:
\begin{equation}
\mathcal{K}_m = \mbox{span}\{\Pi_B \bar{f}, K_\Pi \bar{f}, K_\Pi^2 \bar{f},
                \dots, K_\Pi^m \bar{f} \}.
\end{equation}
This is equivalent to applying the nullspace
method~\cite{BenGolLie99} to the saddle-point system~\cref{sadsys}.

\subsection{Improved tentative interpolation $P_0$ and orthogonal projection $\Pi_B$
with sparsity pattern expansion and sparse $QR$}\label{sec:construct-ptent}

A crucial point for energy minimization interpolation is the availability of a tentative prolongation
$P_0$ that satisfies the near kernel representability constraint~\cref{TSP}. While this is
relatively straightforward for scalar diffusion equations where $V$ has only one column, it is
not trivial for vector-valued PDEs such as elasticity. One specific difficulty is that
while forming the $i$-th constraint equation~\cref{Constr},
we must ensure that $V_c(\mathcal{J}_i,:)$ is full-rank. If it is not full rank, then
no prolongation operator is able to satisfy~\cref{TSP} and 
the solution to~\cref{sadsys} is not possible in general.  We consider two possible remedies:
\begin{enumerate}
\item Add strongly connected neighbors to the pattern of the $i$-th row of $P$ to
   enlarge $\mathcal{J}_i$ until $V_c(\mathcal{J}_i,:)$ is full-rank (i.e., sparsity pattern expansion); or
\item Compute the least square solution of~(\ref{Constr}) as is done in~\cite{OlsSchTum11}.
\end{enumerate}

A novel aspect of this work is our pursuit of sparsity pattern expansion.
We find that this careful construction of the sparsity pattern, which guarantees
that each constraint is exactly satisfied as $V_c(\mathcal{J}_i,:)$ is always full-rank,
greatly improves performance on some problems.

To accomplish this task, we adopt a dynamic-pattern, least-squares fit (LSF) procedure
that satisfies~\cref{TSP} or equivalently~\cref{eq:vec_tent_constraint}.
For each row of $W$ (corresponding to a fine node $i$), this is equivalent to 
satisfying the local dense system~\cref{Constr}. 
For simplicity, we rewrite equation~\cref{Constr} by dropping the row subscript $i$, with $\row{w} = \overline{\rowsubi{w}}$ and $\row{v} = v_i$, yielding 
\begin{equation}\label{TSP2}
   \localB \row{w} = \row{v},
\end{equation}
where $\localB = V_c(\mathcal{J}_i,:)^T$ corresponds to a diagonal block of $B^T$ in~\cref{sadsys}, when
the non-zero entries of $P$ are enumerated row-wise. 

Considering this generic FINE node represented by equation~\cref{TSP2}, if there are a sufficient number of COARSE
node neighbors, then $\localB$ has more columns than rows. Hence if we assume a full-rank $\localB$, then~\cref{TSP2} is an underdetermined
system and can be solved in several ways. In order to have a sparse
solution $\row{w}$, we choose a minimal set of columns of $\localB$ using the
\textit{max vol} algorithm~\cite{Knu85,Gor08} to have the best basis.  Here, we satisfy~\cref{TSP2} exactly. 
We note that a related \textit{max vol} approach to computing C-F splittings is used in \cite{brannick2018optimal}.
\begin{remark}
   We adopt this form of a $QR$ factorization with \textit{max vol}, that is as sparse as possible, in order to improve the complexity of our
   algorithm and quality of $P_0$.  While it is a relatively minor change to the algorithm's structure, we count it as a useful novelty of our efficient implementation.
\end{remark}

If, on the contrary, the number of neighboring COARSE nodes is not sufficient ($\row{v} \notin \mbox{span}(\localB)$), then~\cref{TSP2} cannot be satisfied because it is overdetermined.  This may occur
not only when $\localB$ is \textit{skinny},~i.e., the number of columns is smaller
than the number of rows, but more often because $\localB$ is rank deficient even
if it has a larger number of columns, i.e., it is \textit{wide}.
In elasticity problems and in particular with shell finite elements,
this issue arises often in practice with standard distance one coarsening,
where during the coarsening process, some FINE nodes may occasionally
remain isolated. Our strategy is to gradually increase the interpolation distance
for violating nodes where $\row{v} \notin \mbox{span}(\localB)$, thus widening the interpolatory set (i.e., adding columns to $\localB$) where it is
necessary. \cref{PTENT_setup} describes how to set-up $\widehat{p_0}$, the vector form of initial tentative prolongation $\widehat{P_0}$. (Algorithm~\ref{EMIN-SetUp}, described later, will further process $\widehat{P_0}$ for the final tentative prolongation $P_0$).
\begin{algorithm}[h!]
\caption{\bf Tentative Prolongation Set-Up}
\begin{algorithmic}[1]
\Procedure{PTent\_SetUp}{$S$, $V$, $\lmax$}\\
\hspace*{\algorithmicindent}\textbf{input:} $S$ -- strength of connection matrix\\
\hspace*{\algorithmicindent}$\phantom{\mbox{\textbf{input:}}}$ $V$ -- near kernel modes\\
\hspace*{\algorithmicindent}$\phantom{\mbox{\textbf{input:}}}$ $\lmax$ -- maximum interpolation distance\\
\hspace*{\algorithmicindent}\textbf{output:} $\widehat{p_0}$ -- initial tentative prolongation
\ForAll {FINE nodes $i$}
   \State Set $l = 0$;
   \While {$l < \lmax$}
      \State Set $l = l+1$;
      \State Form $\mathcal{N}_i$ with the strong neighbors of $i$ up to distance $l$;
      \State Select the best columns of $V_c(\mathcal{N}_i,:)^T$ using \textit{max vol}
        to form $\localB$;
      \State Compute the least-squares solution to $\localB \row{w} = \row{v}$;
      \If {$\| \row{v} - \localB \row{w} \|_2 = 0$}
        \State {Assign $\row{w}^T$ as $i$-th row of $\widehat{p_0}$};
        \State {\bf break};
      \EndIf
   \EndWhile
\EndFor
\EndProcedure
\end{algorithmic}
\label{PTENT_setup}
\end{algorithm}

\begin{remark}
Avoiding a \textit{skinny} or rank deficient 
local $\localB$ block is also important for the construction of the orthogonal projection
$\Pi_B = I - B (B^T B)^{-1} B^T$ that maps vectors of $\mathbb{R}^n$ to
$\mbox{Ker}(B^T)$. Note that $\Pi_B$ is used not only to correct the conjugate
gradients search direction, but also to ensure the initial prolongation satisfies
the near kernel constraint. 
In general, the size of $\localB$
during energy minimization is larger than when constructing 
tentative prolongation because of the additional sparsity pattern expansion in (\ref{enlPatt}).  Thus, this ``rank-deficiency'' issue
is ameliorated, but there are pathological cases where it has been observed in practice. 
\end{remark}

We now describe our procedure for computing local blocks of $\Pi_B$, called $\Pi_\localB$, and a single row of
the final tentative prolongation $P_0$, called $\rowtent{w}^T$.
The global procedure to build $P_0$ and
$\Pi_B$ is then obtained by repeating this local algorithm for each FINE row. 
Denote by $\rowtent{\ww}^T$ the \textit{starting} tentative prolongation
row. We use the word \textit{starting}, because in general, we can receive a tentative
prolongation that does not satisfy the constraint. That is, we may have:
\begin{equation}
\label{eqn:Btildestart}
   \localB \rowtent{\ww} \ne \row{v}.
\end{equation}
Denote by $n_l$ and $m_l$ the dimensions of the local system, so that
$\localB \in \mathbb{R}^{n_l \times m_l}$. To fulfill condition~\cref{TSP2}, we must find
a correction to $\rowtent{\ww}$, say $\row{\delta}$, such that:
\begin{equation}
   \localB \row{\delta} = \row{v} - \localB \rowtent{\ww} = \row{r}
\label{resEq}
\end{equation}
and then set:
\begin{equation}
   \rowtent{w} = \rowtent{\ww} + \row{\delta}.
\end{equation}
To enforce condition \cref{constraint} efficiently, we construct an orthonormal basis of
$\mbox{range}(\localB)$, say $\localQ$, that gives rise to the desired local orthogonal projector:
\begin{equation}
\Pi_\localB = I - \localQ \localQ^T.
\end{equation}

If the prolongation pattern is large enough, the vast majority of the above local problems
will be such that $n_l \ge m_l$ with $\localB$ being also full-rank,~i.e.,
$\mbox{rank}(\localB) = m_l$. Thus, an \textit{economy-size} QR decomposition is firstly
performed on $\localB$, $\localB = \localQ R$ with
$\localQ \in \mathbb{R}^{n_l \times m_l}$ and $R \in \mathbb{R}^{m_l \times m_l}$, and
$\localQ$ is used to form the local projector.
Then, through the same QR decomposition, we compute $\row{\delta}$ as the least
norm solution of the underdetermined system~(\ref{resEq}):
\begin{equation}
   \row{\delta} = (\localB^T \localB)^{-1} \localB^T \row{r} = (R^T \localQ^T \localQ R)^{-1} R^T \localQ^T \row{r}= R^{-1} \localQ^T \row{r}.
\end{equation}
Note that any solution to~(\ref{resEq}) would be equivalent, as the optimal choice in
terms of global prolongation energy is later computed by the restricted CG algorithm.
If the initial tentative prolongation arises from the LSF set-up, it should already
fulfill~\cref{resEq} and $\row{r} \equiv 0$. However, in the most difficult cases
even extending the interpolatory set with large distance $l_{\text{max}}$ is not sufficient to guarantee
an exact interpolation of the near kernel for all the FINE nodes.
For these FINE nodes~---~i.e., when  $n_l < m_l$ or $\localB$ is not full-rank~---~we compute an SVD
decomposition of $\localB$:
\begin{equation}
\localB = U \Sigma V^T.
\end{equation}
From the diagonal of $\Sigma$, we determine the rank of $\localB$, say $k_l$, and use the
first $k_l$ columns of $U$ to form $\localQ$ and thus $\Pi_\localB$. Finally, since in
this case it could be impossible to satisfy the constraint because the system is
overdetermined, we use the least square solution to~\cref{resEq} to compute the
correction for $\rowtent{\ww}$:
\begin{equation}
   \row{\delta} = V \Sigma^{\dagger} U^T \row{r}.
\end{equation}
It is important to recognize that for these specific FINE nodes, energy minimization
cannot reduce the energy, because the constraint does not leave any degree of freedom.
Consequently, this situation should be avoided when selecting the COARSE
variables and the prolongation pattern, because a prolongation violating the near kernel
constraint will likely fail in representing certain smooth modes.
The pseudocode to set-up $\Pi_B$ and correct the initial tentative prolongation
$P_0$ (after~\cref{PTENT_setup}) is provided in~\cref{EMIN-SetUp}.

\begin{algorithm}
\caption{\bf Energy Minimization Set-Up}
\begin{algorithmic}[1]
\Procedure{EMIN\_SetUp}{$B$, $g$, $\widehat{p_0}$}\\
\hspace*{\algorithmicindent}\textbf{input:} $B$ -- block diagonal constraint matrix, as in equation (\ref{eq:vec_tent_constraint}) \\
\hspace*{\algorithmicindent}$\phantom{\mbox{\textbf{input:}}}$ $g$ -- constraint right-hand-side\\
\hspace*{\algorithmicindent}$\phantom{\mbox{\textbf{input:}}}$ $\widehat{p_0}$ -- initial tentative prolongation\\
\hspace*{\algorithmicindent}\textbf{output:} $\Pi_B$ -- projection matrix, constructed block-wise\\
\hspace*{\algorithmicindent}$\phantom{\mbox{\textbf{output:}}}$ $p_0$ -- final tentative prolongation
\ForAll {FINE nodes $i$}
   \State Gather $\localB$, $\row{v}$ and $\rowtent{\ww}$ for row $i$, as in equation (\ref{eqn:Btildestart});
   \State Compute $\row{r} = \row{v} - \localB \rowtent{\ww}$;
   \State FAIL\_QR = false;
   \If {$n_l \ge m_l$}
       \State Compute economy-size QR of $\localB$: $\localB = \localQ R$;
       \If {$\mbox{rank}(R) < m_l$}
           \State Set FAIL\_QR = true;
       \Else
           \State Compute $\row{\delta} = R^{-1} \localQ^T  \row{r}$;
       \EndIf
   \EndIf
   \If {$n_l < m_l$ or FAIL\_QR}
       \State Compute SVD of $\localB$: $\localB = U \Sigma V^T$;
       \State Determine $k = \mbox{rank}(\localB)$ using the diagonal of $\Sigma$;
       \State Set $\localQ = U(:,1:k)$;
       \State Compute $\row{\delta} = V \Sigma^{\dagger} U^T \row{r}$;
   \EndIf
   \State Compute $\rowtent{w} = \rowtent{\ww} + \row{\delta}$ and insert $\rowtent{w}^T$ as row $i$ of $p_0$;
   \State Set $i$-th block of $\Pi_B$ as $I - \localQ \localQ^T$;
\EndFor
\EndProcedure
\end{algorithmic}
\label{EMIN-SetUp}
\end{algorithm}
%


\subsection{Improved stopping criterion and energy monitoring for CG-based energy minimization}

Stopping criteria plays an important role in the overall cost and effectiveness
of energy minimization. Here, we introduce a measure for monitoring energy and 
halting in the algorithm.

Since CG often converges quickly for energy minimization, it is
common to fix the number of iterations in advance~\cite{ManBreVan99,OlsSchTum11}. However, in the case of a
more challenging problem, several iterations may be needed, thus requiring an
accurate stopping criterion. One immediate option is to use the relative
residual, yet this may not be a close indicator of \textit{energy}. In the
following, we analyze CG for a generic $Ax = b$, however our observations extend to
PCG as well.

In the CG algorithm, once the search direction $p_k$ is defined at iteration $k$, the
scalar $\alpha$ is computed:
\begin{equation}\label{eq:CG_a}
\alpha_k = \frac{p_k^T r_k}{p_k^T A p_k},
\end{equation}
in such a way that the new approximation $x_{k+1} = x_k + \alpha_k p_k$ minimizes the square
of the energy norm of the error, i.e.,
\begin{equation}\label{eq:CG_E}
E_k = (x_k - h)^T A (x_k - h),
\end{equation}
where $h = A^{-1}b$ is the true solution. The difference in energy $\Delta E_{k+1} =
E_{k+1} - E_k$ between two successive iterations $k$ and $k+1$ can be computed as
\begin{equation}\label{eq:CG_dE}
\begin{split}
\Delta E_{k+1} &= (x_k + \alpha_k p_k)^T A (x_k + \alpha_k p_k) - 2 b^T (x_k +
\alpha_k p_k) - x_k^T A x_k + 2 b^T x_k \\
&= 2 \alpha_k x_k^T A p_k + \alpha_k^2 p_k^T A p_k - 2\alpha_k b^T p_k = - 2 \alpha_k
r_k^T p_k + \alpha_k^2 p_k^T A p_k\\
&= \alpha_k (-2 r_k^T p_k + \alpha_k p_k^T A p_k) = -\alpha_k (p_k^T r_k) = -\frac{(p_k^T
r_k)^2}{p_k^T A p_k} < 0.
\end{split}
\end{equation}
From~\cref{eq:CG_a} and~\cref{eq:CG_dE}, it is possible to measure, with minimal cost,
the energy decrease provided by the $(k+1)$-st iteration. Indeed, by noting that
$\alpha_k$ is computed as the ratio between the two values $\alpha_{\text{num}} = p_k^T r_k$
and $\alpha_{\text{den}} = p_k^T A p_k$, the energy decrease reads
\begin{equation}\label{eq:CG_dEa}
\Delta E_{k+1} = \frac{\alpha_{\text{num}}^2}{\alpha_{\text{den}}}.
\end{equation}
The relative value of the energy variation with
respect to the initial variation (first iteration) is monitored and convergence is
achieved when energy is sufficiently reduced:
\begin{equation}
\frac{\Delta E_k}{\Delta E_1} \le \tau
\label{relRed}
\end{equation}
for a small user-defined $\tau$.

\subsection{Improved preconditioning for CG-based energy minimization}
\label{sec:precon}


Before introducing PCG, we present some important properties of matrices $K$ and $B$,
that we leverage in the design of effective preconditioners for energy minimization.
In particular, we will see that, thanks to the non-zero structures of $K$ and $B$,
point-wise Jacobi or Gauss-Seidel iterations prove particularly effective
in preconditioning the projected block $\Pi_B K \Pi_B$.

Let us assume that the vector form of prolongation $p$ has been 
obtained by collecting the non-zeroes of $P$ \emph{row-wise}, so that $B$ is block diagonal.
Denoting by $Q$ the matrix collecting an orthonormal basis of $\mbox{range}(B)$, and $Z$
an orthonormal basis of $\mbox{ker}(B)$, by construction we have
\begin{equation}
[Q \;\; Z] [Q \;\; Z]^T = [Q \;\; Z]^T [Q \;\; Z] = I,
\end{equation}
i.e., the matrix $[Q \;\; Z]$ is square and orthogonal. Moreover, since $B$ is block
diagonal, both $Q$ and $Z$ are block diagonal and can be easily computed and stored.
We note that, by construction, each column of $Q$ and $Z$ refers to a specific row
of the prolongation, and, due to the block diagonal pattern chosen for $B$,
also each column of $Q$ and $Z$ is non-zero only in the positions corresponding
to the entries of $p$ collecting the non-zeroes of a specific row of $P$, as is
schematically shown in Fig.~\ref{fig:PBQZ_scheme}. More precisely, using the same notation as
in~\cref{Constr}, let us define $\LL_i$ as the set of indices in $p$ corresponding
to the non-zero entries in the $i$-th row of $P$ so that:
\begin{equation}
p(\LL_i) = P(i,\J_i)
\end{equation}
and define $\J_{B,i}$ and $\J_{Z,i}$ as the set of columns of $B$ and $Z$, referring
to the $i$-th row of $P$. Then,
\begin{equation}
\left.
\begin{array}{l}
B(k,\J_{B,i}) = 0 \\
Q(k,\J_{B,i}) = 0 \\
Z(k,\J_{Z,i}) = 0
\end{array}
\right\} \quad \mbox{if} \quad k \notin \LL_i .
\end{equation}

\begin{figure}[!h]
 \centering
 \includegraphics[width=0.60\textwidth]{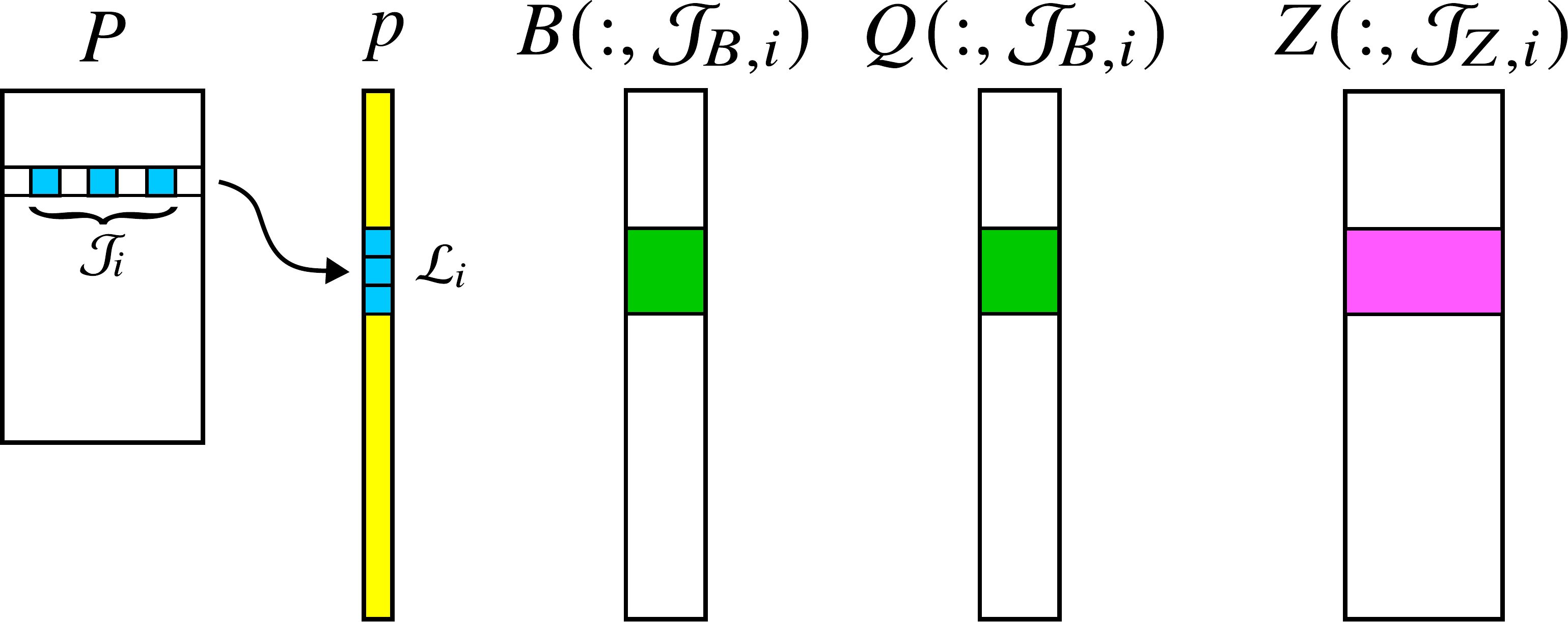}
 \caption{Non-zero pattern of the columns of $B$, $Q$ and $Z$ corresponding to a specific
          row of the prolongation $P$.}
 \label{fig:PBQZ_scheme}
\end{figure}

We are now ready to state two theorems that will be useful in explaining the choice of our
preconditioners.

\begin{theorem}
\label{thr:diagZKZ}
The diagonal of the projected matrix $Z^T K Z$ is equal to the projection of the diagonal
of $K$:
\begin{equation}
\diag(Z^T K Z) = \diag(Z^T D_K Z),
\label{eq:diagZKZ}
\end{equation}
where $D_K$ is the matrix collecting the diagonal entries of $K$. Moreover, $Z^T D_K Z$ is a
diagonal matrix.
\end{theorem}

\begin{proof}
Let us consider the block of columns of $Z$ relative to row $i$, that is $Z(:,J_{Z,i})$,
remembering that it is non zero only for the row indices in $\LL_i$. As a consequence,
the square block $H_i$ obtained by pre- and post-multiplying $K$ by $Z(:,J_{Z,i})$ is
computed as:
\begin{equation}
H_i(j,k) = \sum_{r \in \LL_i} \left( \sum_{s \in \LL_i} Z(s,k) K(r,s) \right) Z(r,j)
\quad \mbox{for} \quad j,k \in \J_{Z,i} .
\end{equation}
However, $K(r,s)$ for $r,s \in \LL_i$ represents the connection between $P(i,j_r)$
and $P(i,j_s)$, which is non-zero only for $j_r = j_s$,~i.e., $r = s$, because in $K$
there is no connection between different columns of $P$. Moreover, due to~\cref{minblk},
$K(r,r) = A(i,i)$ for every $r \in \LL_i$. As the columns of $Z$ are orthonormal by
construction, $Z^T Z = I$, it immediately follows that the square block $H_i$ is diagonal
with all its non zero entry equal to $A(i,i)$. The fact that $Z^T D_K Z$ is a diagonal
matrix follows from the above observation that
$K(\LL_i,\LL_i) = D_K(\LL_i,\LL_i) = A(i,i) I_m$, with $I_m$ the identity matrix having
size $m$ equal to the cardinality of $\LL_i$.
\end{proof}

\begin{corollary}
\label{thr:nullZKQ}
The product $Z^T D_K Q$, where $D_K$ is defined as in Theorem \ref{thr:diagZKZ}, is
equal to the null matrix:
\begin{equation}
Z^T D_K Q = 0.
\label{eq:nullZKQ}
\end{equation}
\end{corollary}

\begin{proof}
Due to the block-diagonal structure of $Q$ and $Z$, all the off-diagonal blocks
   of $Z^T D_K Q$ are empty. Then, equation~\cref{eq:nullZKQ} follows from $D_K(\LL_i,\LL_i) = A(i,i) I_m$
and $Z(:,\J_i) \perp Q(:,\J_i)$.
\end{proof}

\subsubsection{Preconditioned CG-based Energy Minimization}

\begin{algorithm}[t!]
\caption{\bf Preconditioned Conjugate Gradients for Energy Minimization}
\begin{algorithmic}[1]
\Procedure{EMIN\_PCG}{maxit, $\tau$, $K$, $f$, $\Pi_B$, $M$, $P_0$}\\
\hspace*{\algorithmicindent}\textbf{input:} maxit -- maximum iterations \\
\hspace*{\algorithmicindent}$\phantom{\mbox{\textbf{input:}}}$ $\tau$ -- energy convergence tolerance\\
\hspace*{\algorithmicindent}$\phantom{\mbox{\textbf{input:}}}$ $K$ -- system matrix (applied matrix-free with $A$) \\
\hspace*{\algorithmicindent}$\phantom{\mbox{\textbf{input:}}}$ $f$ -- right-hand-side $f$ from equation \ref{11blk} \\
\hspace*{\algorithmicindent}$\phantom{\mbox{\textbf{input:}}}$ $\Pi_B$ -- projection matrix\\
\hspace*{\algorithmicindent}$\phantom{\mbox{\textbf{input:}}}$ $M$ -- preconditioner\\
\hspace*{\algorithmicindent}$\phantom{\mbox{\textbf{input:}}}$ $P_0$ -- tentative prolongation\\
\hspace*{\algorithmicindent}\textbf{output:} $P$ -- final prolongation

\State Extract global weight vector $w_0$ row-wise from $P_0$
\State $\Delta w = 0$;
\State $r = f - \Pi_B K w_0$;
\For {$k = 1,\dots,$maxit}
  \State $z = \Pi_B M^{-1} r$\label{algo:PCG:res};
  \State $\gamma = r^T z$;
  \If{$i = 1$}
    \State $y = z$;
  \Else
    \State $\beta = {\gamma}/{\gamma_{old}}$;
    \State $y = z + \beta y$;
  \EndIf
  \State $\gamma_{old} = \gamma$;
  \State $\breve{y} = \Pi_B K y$;
  \State $\alpha = \gamma / (y^T \breve{y}$);
  \State $\Delta E_k = \gamma \alpha$
  \State \textbf{if} {$\Delta E_k < \tau \Delta E_1$} {\Return}
  \State $\Delta w = \Delta w + \alpha y$;
  \State $r = r - \alpha \breve{y}$;
\EndFor
\State $w = w_0 + \Delta w$;
\State Form final prolongation $P = [W; I]$ with global weight vector $w$
\EndProcedure
\end{algorithmic}
\label{algo:PCG}
\end{algorithm}
Preconditioning CG can greatly improve convergence, but special care should be
taken to maintain the search direction $y$ in the space of vectors satisfying the
near kernel constraint. In other words, $y$ must satisfy $\Pi_B y \equiv y$.
In~\cite{OlsSchTum11}, a Jacobi preconditioner is adopted that satisfies this requirement,
but, due to the special properties of the matrix $K$, it is possible to compute a more
effective preconditioner. Denoting by $M^{-1}$ any approximation of $K^{-1}$,
we use $\Pi_B M^{-1} \Pi_B$ to precondition $\Pi_B K \Pi_B$ in order to guarantee
the constraint. The resulting PCG algorithm is outlined in~\cref{algo:PCG}, where,
since $\Pi_B$ is a projection, we can avoid premultiplying $r$ by $\Pi_B$
(line \ref{algo:PCG:res}) as $r$ already satisfies the constraint.

In the remainder of this section we focus our attention on $Z^T K Z$ instead of
$\Pi_B K \Pi_B$, because, as $\Pi_B = I - Q Q^T = Z Z^T$, they have the same spectrum.
Our aim is to find a good preconditioner for $Z^T K Z$. Unfortunately,
although $K$ is block diagonal and several effective preconditioners can be easily built
for it, $Z^T K Z$ is less manageable and further approximations are needed.

By pre- and post-multiplying $K$ by $[Z\; Q]^T$ and $[Z\; Q]$, respectively, we can write
the following $2 \times 2$ block expression:
\begin{equation}
[Z\; Q]^T K [Z\; Q] =
\begin{bmatrix}
Z^T K Z & Z^T K Q \\
Q^T K Z & Q^T K Q \\
\end{bmatrix}\!,
\end{equation}
and, since we are interested in the inverse of $Z^T K Z$, we can express it as the Schur
complement of the leading block of the inverse of $[Z\; Q]^T K [Z\; Q]$
\cite[Chapter~3]{vassilevski2008multilevel3}:
\begin{equation}
([Z\; Q]^T K [Z\; Q])^{-1} = [Z\; Q]^T K^{-1} [Z\; Q] =
\begin{bmatrix}
Z^T K^{-1} Z & Z^T K^{-1} Q \\
Q^T K^{-1} Z & Q^T K^{-1} Q \\
\end{bmatrix}\!,
\end{equation}
from which it follows that:
\begin{equation}
(Z^T K Z)^{-1} = Z^T K^{-1} Z - Z^T K^{-1} Q (Q^T K^{-1} Q)^{-1} Q^T K^{-1} Z.
\label{eq:invSch}
\end{equation}
When $K^{-1}$ is approximated with the inverse of the diagonal of $K$, $M_J = \diag(K)$,
because of Theorem~\ref{thr:nullZKQ}, we have that $Z^T M_J^{-1} Q = 0$, and the expression
\eqref{eq:invSch} becomes:
\begin{equation}
(Z^T K Z)^{-1} \simeq M_1^{-1} = Z^T M_J^{-1} Z,
\end{equation}
which corresponds to a Jacobi preconditioning of $Z^T K Z$.

We highlight that only for Jacobi can the post-multiplication by $\Pi_B$
be neglected in line \ref{algo:PCG:res}, since $M_J^{-1}$ does not introduce components along $\mbox{range}(Q)$.
This is consistent with~\cite{OlsSchTum11}, where the Jacobi
preconditioner is used, and no post-multiplication by $\Pi_B$ is adopted.

If a more accurate preconditioner is needed, $K^{-1}$ can be approximated using a
block-wise symmetric Gauss-Seidel (SGS) iteration, that is:
\begin{equation}
K^{-1} \simeq M^{-1}_{SGS} = (L+D)^{-T} D (L+D)^{-1} ,
\end{equation}
which substituted into equation~\eqref{eq:invSch} reads:
\begin{equation}
(Z^T K Z)^{-1} \simeq M_2^{-1} = Z^T M^{-1}_{SGS} Z -
Z^T M^{-1}_{SGS} Q (Q^T M^{-1}_{SGS} Q)^{-1} Q^T M^{-1}_{SGS} Z.
\label{eq:SGS_prec}
\end{equation}
Since the application of \eqref{eq:SGS_prec} is still impractical due to the presence of
the term $(Q^T M^{-1}_{SGS} Q)^{-1}$, we neglect the second member of the right-hand side,
based on the heuristic that $Z^T M^{-1}_{SGS} Q$ should be small, because $Z^T M_J^{-1}
Q = 0$. After this simplification, we obtain the final expression of the projected SGS
preconditioner:
\begin{equation}
K^{-1} \simeq M_2^{-1} = Z^T M^{-1}_{SGS} Z = Z^T (L+D)^{-T} D (L+D)^{-1} Z.
\end{equation}
Note that while $M_1^{-1}$ is exactly the Jacobi preconditioner of
$Z^T K Z$, $M_2^{-1}$ is only an approximation of the exact SGS preconditioner of $Z^T K Z$. However, we
will show in the numerical results that it is able to significantly accelerate
convergence.

\section{Efficient parallel implementation}


Energy minimization has historically been considered computationally expensive for AMG setup and
real applications. Nevertheless, a cost effective implementation is still possible,
but requires special care 
with the algorithm parallelization.
In this work, we build our AMG preconditioner with Chronos
\cite{IsoFriSpiJan21,CHRONOS-webpage}, which provides the standard numerical kernels
usually required in a distributed memory AMG implementation such as the communication of
ghost unknowns, coarse grid construction (e.g., computing the PMIS C-F partition \cite{de2006reducing}), sparse matrix-vector, and sparse matrix-matrix
product, etc. For energy minimization, however, we developed three
specific kernels that are not required in other AMG approaches, but are critical for an
efficient parallel implementation here, i.e., the
sparse product between $K$ and $p$, application of the projection $\Pi_B$, and
symmetric Gauss-Seidel with $K$. We do not list Jacobi preconditioning,
because it simply consists of a row-wise scaling of $P$.

The first issue related to the product by $K$ is that in practical applications $K$ cannot
be stored. In fact, if we consider a prolongation $P$ having $r$ non-zeroes per row
on average, the number of non-zeroes per column will be approximately
$s \simeq \frac{n_f}{n_c} r$ and the $K$ matrix would be of size $n_c s \times n_c s$
with about $n_c s \, t \simeq n_f r \, t$ non-zeroes to be stored, where $t$ is the average
number of non-zeroes per row of $A$. Often, $r \simeq t$ and storing $K$ becomes several times
more expensive than storing $A$. For instance in practical elasticity problems, $K$ can
be up to 20 times larger than $A$, making it unavoidable to proceed in
matrix-free mode for $K$. Fortunately, the special structure of $K$ allows us to interpret
the product $K p$ as a sparse matrix-matrix (SpMM) product between $A$ and $P$, but with a fixed,
prescribed pattern on $P$. This property can be easily derived from the definition of
$K$ \cref{11blk} and the vector form of the prolongation $p$.
One advantage of prescribing a fixed pattern for $P$ is that the amount of data exchanged
in SpMM is greatly reduced. First of all, the sparsity pattern adjacency information of $P$ can be
communicated only once before entering PCG, then for all the successive $A P$ products
only the entries of $P$ are exchanged. Moreover, all the buffers to receive and
send messages through the network, can be allocated and set-up only once at the
beginning and removed at the end of the minimization process.  In practice, we find that
these optimizations reduce the cost of SpMM by about 50\%.

By contrast, the construction and application of $\Pi_B$ does not require any communication. In fact,
we distribute the prolongation row-wise among processes, and since $B$ is
block diagonal, with each block corresponding to a row of $P$ as shown in \eqref{Constr_Mat},
we distribute $B$ blocks to the process owning the respective row of $P$. Following this scheme, each block $\localB$ of $B$ is
factorized independently to obtain $Q$ which is efficiently processed by the LAPACK routine DGEMV
when applying $\Pi_B$.

Finally, we emphasize that there is no parallel bottleneck when parallelizing the
application of symmetric Gauss-Seidel with $K$.  This is because $K$ is block diagonal and,
at least in principle, each block can be assigned to a different process. However
as in the $K p$ product, the main issue is the large size of $K$ which prevents
explicit storage. As above, we rely on the equivalence between the $K p$
product and the $A P$ product with a prescribed sparsity pattern. The symmetric Gauss-Seidel step
is then performed matrix-free, similar to the $A P$ product, by saving again a considerable
amount of data exchange because it is not necessary to communicate adjacency
information or reallocate communication buffers. As expected, the symmetric Gauss-Seidel application exhibits
a computational cost comparable to that of the $A P$ product.

\section{Numerical experiments}
\label{sec:numerics}

In this section, we investigate the improved convergence offered by energy minimization
and the resulting computational performance and parallel efficiency for real-world
applications. This numerical section is divided into three parts: a detailed analysis
on the prolongation energy reduction and related AMG convergence for two
medium-size matrices, a weak scalability test for an elasticity problem on a uniformly
refined cube, and a comparison study for a set of large real world matrices
representing fluid dynamics and mechanical problems.

As a reference for the proposed approach, we compare
with the well-known and open source solver GAMG \cite{balay2019petsc}, a smoothed aggregation-based
AMG code from PETSc.  In both cases, Chronos and GAMG are used with preconditioned conjugate gradients (PCG).
The GAMG set-up is tuned for each problem starting from its default parameters, as suggested by
the user guide \cite{PetscGuide}. As smoother we have chosen the most powerfull option available in Chronos and GAMG, that is FSAI and Chebyshev accelerated Jacobi, respectively.
Each time such default parameters are modified, we report the chosen values in the relevant section. Since all of the experiments are run in parallel,
the system matrices are partitioned with ParMETIS~\cite{parmetis} before the
AMG preconditioner set-up to reduce communication overhead.

All numerical experiments have been run on the Marconi100 supercomputer located in the
Italian consortium for supercomputing (CINECA). Marconi100 consists of 980 nodes based on
the IBM Power9 architecture, each equipped with two 16-core IBM POWER9 AC922 @3.1 GHz
processors. For each test, the number of nodes, $N$, is selected so that each node has approximately
640,000,000 nonzero entries, and, consequently, the number of nodes is problem dependent.
Each node is always fully exploited by using 32 MPI tasks, i.e., each task (core) has an average load of
20,000,000 nonzero entries. The number of cores is denoted $n_{cr}$. Only during the
smaller cases in the weak scalability analysis are nodes partially used (i.e., with
less than 32 MPI tasks). Even though Chronos can exploit hybrid MPI-OpenMP parallelism, for
the sake of comparison, it is convenient to use
just one thread, i.e., pure MPI parallelism. Moreover, for such a high load per core, we do not find that
fine-grained OpenMP parallelism is of much help.

The numerical results are presented in terms of total number of computational cores used,
$n_{cr}$, the grid and operator complexities, $C_{gd}$ and $C_{op}$, respectively, the
number of iterations, $n_{it}$, and the setup, iteration, and total times, $T_p$, $T_s$,
and $T_t$ = $T_p$ + $T_s$, respectively. For all the test cases, the right-hand side is a
unit vector. The linear systems are solved with PCG and
a zero initial guess, and convergence is achieved when the $\ell_2$-norm of the
iterative residual drops below 8 orders of magnitude with respect to the $\ell_2$-norm of the
right-hand side.

\subsection{Analysis of the energy minimization process}

We use two matrices for studying prolongation energy reduction, \texttt{Cube}
and \texttt{Pflow742} \cite{paludetto2019novel}. While the former is quite simple, as it
is the fourth refinement level of the linear elasticity cube used in the weak scalability
study, the latter arises from a 3D simulation of the pressure-temperature field in a
multilayered porous medium discretized by hexahedral finite elements. The main source of
ill-conditioning here is the large contrast in the material properties for different layers.
The dimensions of \texttt{Cube} are 1,778,112 rows and 78,499,998 entries, with 44.15
entries per row on average.
The size of \texttt{Pflow742} is 742,793 rows and 37,138,461 entries, for an average
entry-per-row ratio of 50.00.

The overall solver performance is compared against the energy reduction for the fine-level $P$ when
using the energy mininimization Algorithm \ref{algo:PCG} with either Jacobi and Gauss-Seidel as the preconditioner. Results are
analyzed in terms of computational costs and times.
The main algorithmic features we want to analyze are:
\begin{itemize}
  \item how the prolongation energy reduction affects AMG convergence; and
  \item the effectiveness of the preconditioner, i.e., Jacobi or Gauss-Seidel.
\end{itemize}
The energy minimization iteration count (\cref{algo:PCG}) is denoted by $n_{it}^E$.
Between brackets, we also report the relative energy reduction that is used to monitor
the restricted PCG convergence of Algorithm \ref{algo:PCG}, as shown in equation~(\ref{relRed}).
$T_i$ is the time spent to improve the prolongation, with either classical prolongation smoothing with weighted-Jacobi (SMOOTHED) or the energy minimization process. Note that $T_i$ is only part of $T_p$.

Table \ref{tab:sensCube} shows the results for \texttt{Cube}. First, it can be seen
that the energy minimization algorithm produces prolongation operators which lead to somewhat lower
complexities than for the smoothed case. Moreover, energy minimization builds more effective
prolongation operators overall, since the global iteration count ($n_{it}$) is lower. As the
energy minimization iteration count increases, we can observe that $n_{it}$ decreases, while
the setup time ($T_p$) increases. For this simple problem, the optimal point in terms of total
time ($T_t$) is reached using 2 iterations of Jacobi (EMIN-J). Fig. \ref{fig:enerCube} further shows that
2 iterations of Jacobi (EMIN-J) already reaches close to the
achievable energy minimum. Fig. \ref{fig:enerCubeT} compares the cost in wall-clock seconds for each energy minimization
iteration using different preconditioners. For this case and implementation, Jacobi is more efficient.

\begin{table}
  \centering
  \small
  \caption{Analysis of energy minimization for the \texttt{Cube} problem.}
  \label{tab:sensCube}
  \begin{tabular}{lrrrrrrrrr}
    \toprule
    Prolongation & $n_{it}^E$ & $\frac{\Delta E_k}{\Delta E_1}$ & $C_{gr}$ & $C_{op}$ & $n_{it}$ & $T_p$ [s] & $T_s$ [s] &
$T_t$ [s] & $T_i$ [s]\\
    \midrule
    SMOOTHED &-& - & 1.075 & 1.648 & 58 & 52.1 & 13.4 & 65.5 & 6.3 \\
    EMIN-J & 1 & $10^{0}$ & 1.075 & 1.592 & 54 & 47.6 & 11.2 & 58.8 & 11.3 \\
    EMIN-J & 2 & $2 \cdot 10^{-1}$ & 1.075 & 1.589 & 32 & 50.9 & 6.7 & 57.6 & 14.3 \\
    EMIN-J & 4 & $10^{-2}$ & 1.075 & 1.589 & 27 & 57.3 & 5.7 & 63.0 & 20.0 \\
    EMIN-GS & 1 & $10^{0}$ & 1.075 & 1.587 & 28 & 55.5 & 6.0 & 61.5 & 19.3 \\
    EMIN-GS & 2 & $2 \cdot 10^{-2}$ & 1.075 & 1.588 & 26 & 65.6 & 5.5 & 71.1 & 28.8 \\
    EMIN-GS & 4 & $1 \cdot 10^{-4}$ & 1.075 & 1.589 & 26 & 84.4 & 5.6 & 90.0 & 47.7 \\
    \bottomrule
  \end{tabular}
\end{table}

\begin{figure}
  \centering
  \null\hfill
  \subfloat[\texttt{Cube} case\label{fig:enerCube}]
    {\includegraphics[width=0.45\linewidth]{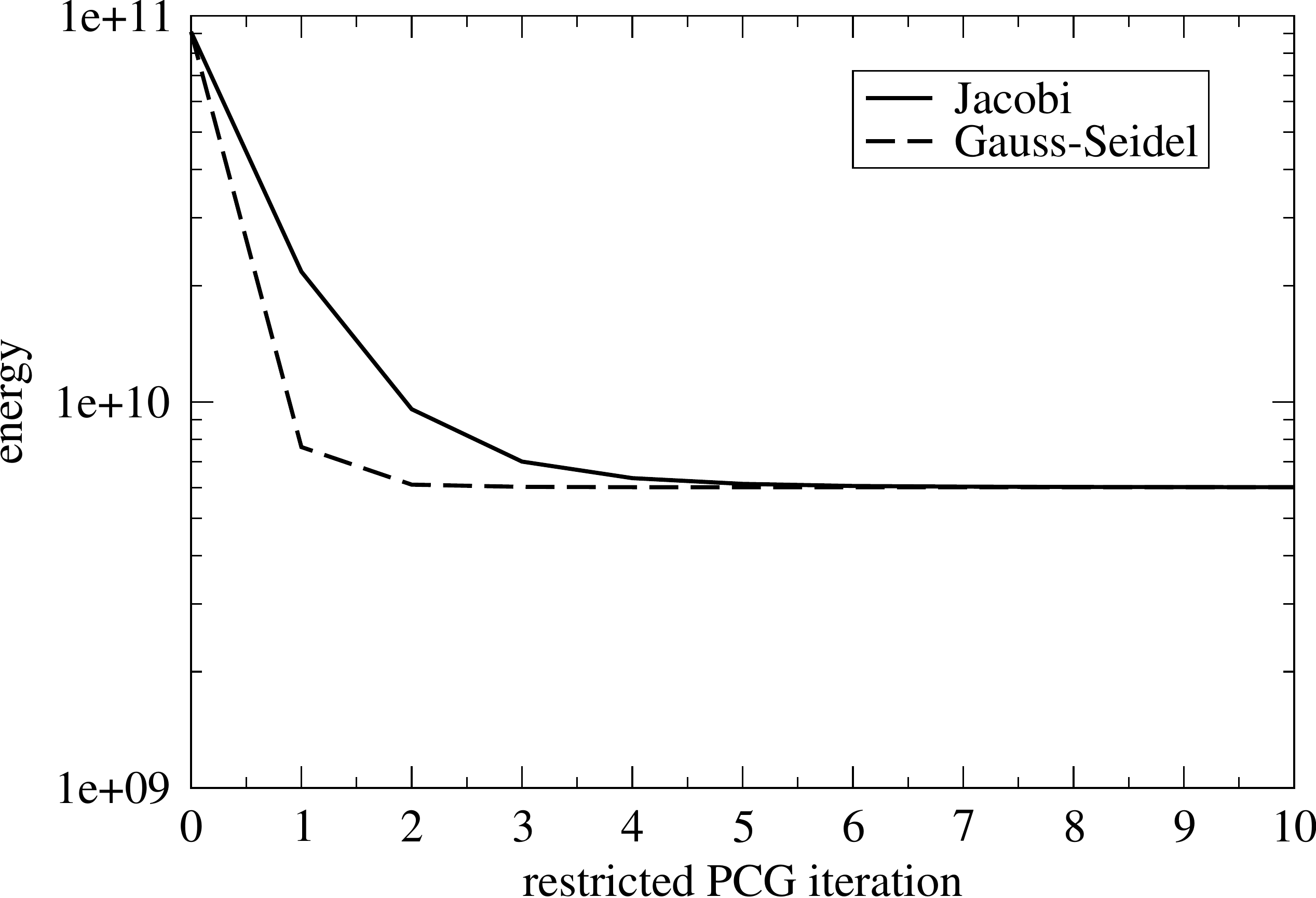}}\hfill
  \subfloat[\texttt{Pflow742} case\label{fig:enerFlow}]
    {\includegraphics[width=0.45\linewidth]{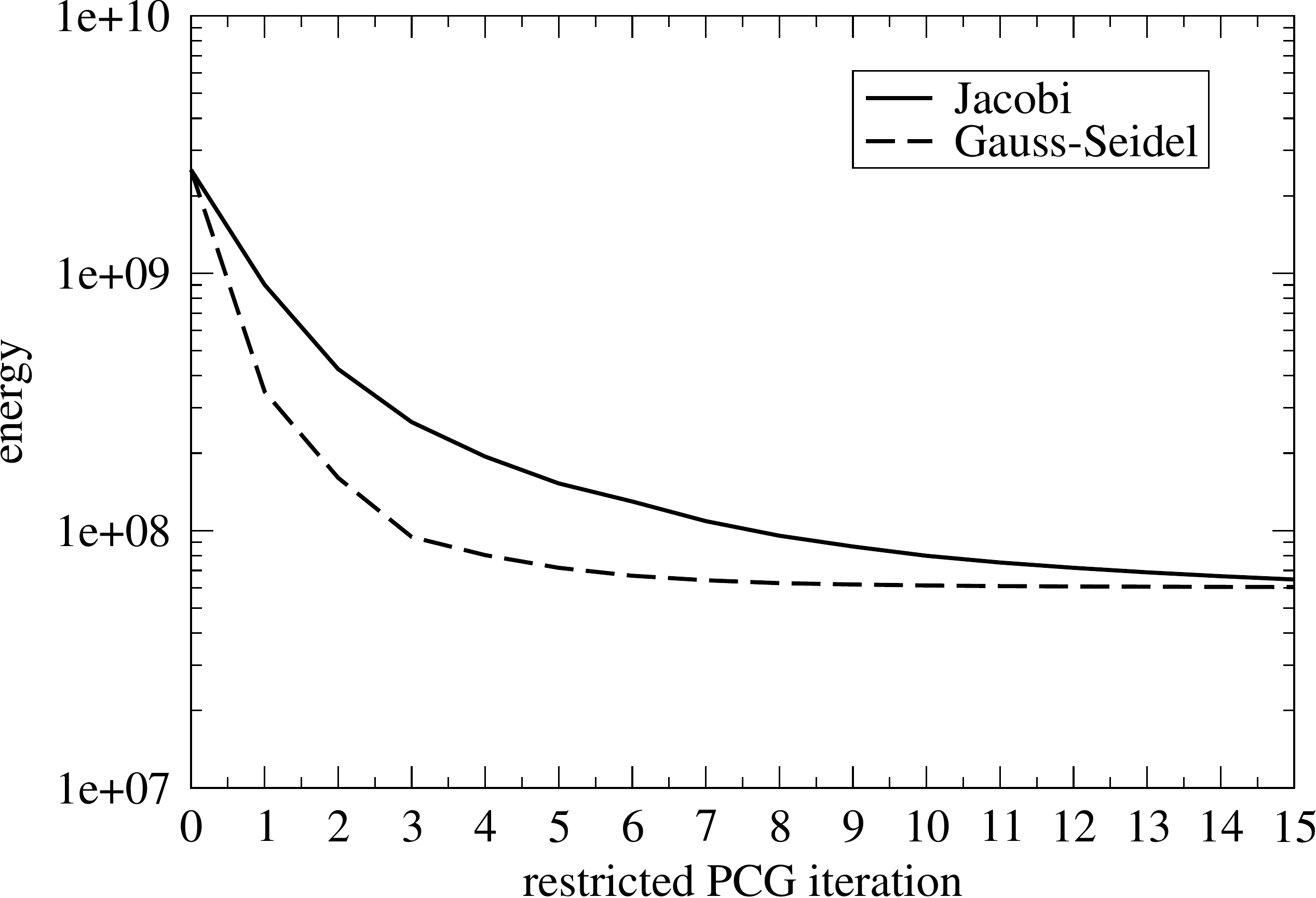}}
  \hfill\null
  \caption{Energy reduction vs. energy minimization iterations with Jacobi and Gauss-Seidel.}
  \label{fig:enerIter}
\end{figure}

\begin{figure}
  \centering
  \null\hfill
  \subfloat[\texttt{Cube} case\label{fig:enerCubeT}]
    {\includegraphics[width=0.45\linewidth]{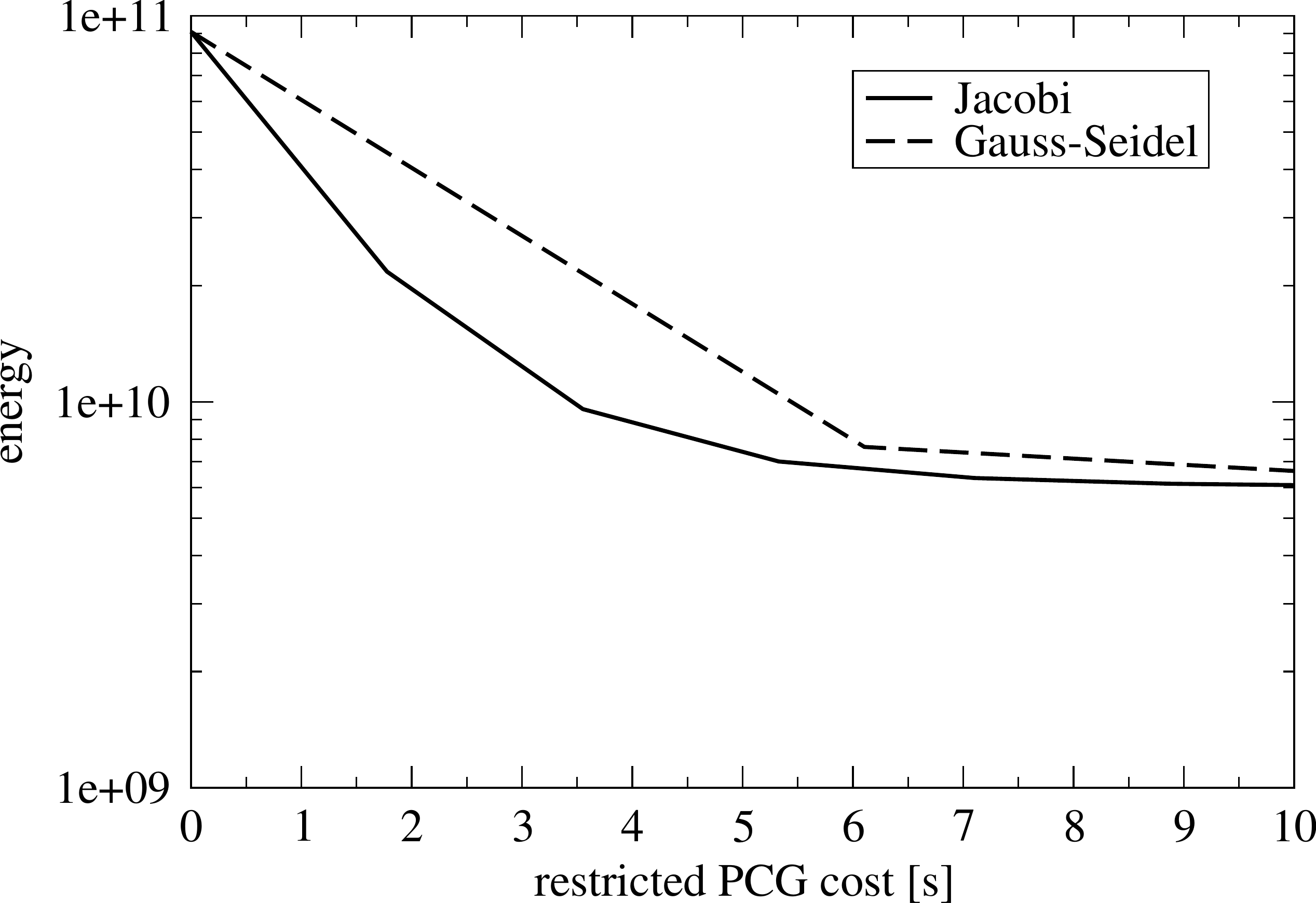}}\hfill
  \subfloat[\texttt{Pflow742} case\label{fig:enerFlowT}]
    {\includegraphics[width=0.45\linewidth]{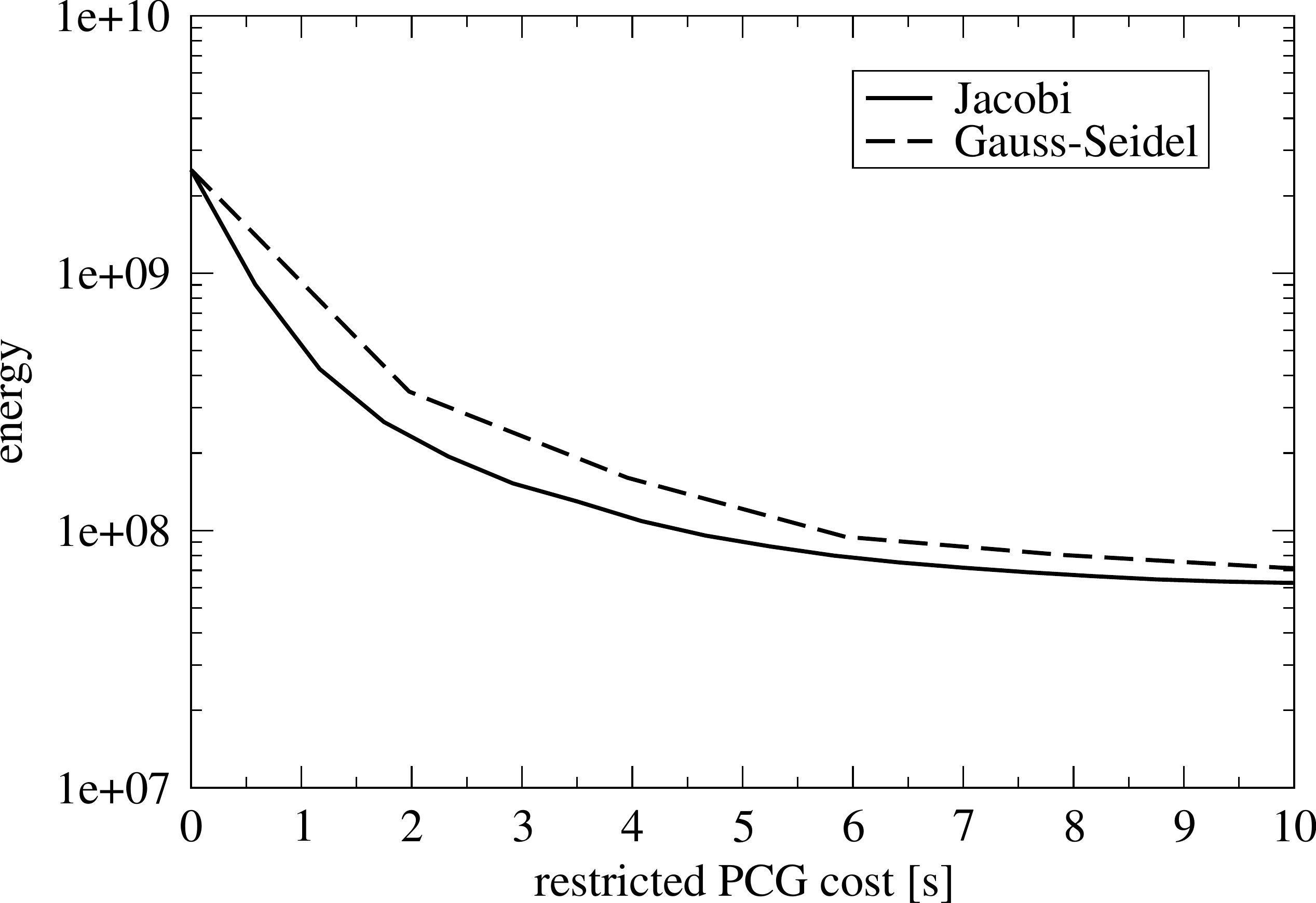}}
  \hfill\null
  \caption{Energy reduction vs. computational cost for energy minimization preconditioned with Jacobi and
           Gauss-Seidel.}
  \label{fig:enerTime}
\end{figure}

Similar conclusions can be drawn for the \texttt{Pflow742} case, as $n_{it}$
monotonically decreases as the energy associated with the prolongation operator is reduced.
As reported by Table \ref{tab:sensFlow}, the optimal total time ($T_t$) is obtained
with 4 Jacobi iterations (EMIN-J). Fig. \ref{fig:enerFlow} shows how the energy of the prolongation operator decreases slower
than for \texttt{Cube} and that Jacobi converges significantly slower than Gauss-Seidel.
However as reported by Fig.  \ref{fig:enerFlowT}, the cost of Gauss-Seidel is still more than that of Jacobi,
although the performance difference is smaller than for \texttt{Cube}.

\begin{table}
  \centering
  \small
  \caption{Analysis of energy minimization for the \texttt{Pflow742} problem.}
  \label{tab:sensFlow}
  \begin{tabular}{lrrrrrrrrr}
    \toprule
    Prolongation & $n_{it}^E$ & $\frac{\Delta E_k}{\Delta E_1}$ & $C_{gr}$ & $C_{op}$ & $n_{it}$ & $T_p$ [s] & $T_s$ [s] &
$T_t$ [s] & $T_i$ [s]\\
    \midrule
    SMOOTHED &-& - & 1.061 & 1.465 & 369 & 23.0 & 29.8 & 52.8 & 2.3 \\
    EMIN-J & 1 & $10^{0}$ & 1.061 & 1.339 & 377 & 20.4 & 27.3 & 47.7 & 2.9 \\
    EMIN-J & 2 & $3 \cdot 10^{-1}$ & 1.061 & 1.344 & 270 & 21.5 & 19.6 & 41.1 & 3.7 \\
    EMIN-J & 4 & $4 \cdot 10^{-2}$ & 1.062 & 1.352 & 219 & 23.3 & 15.7 & 39.0 & 5.4 \\
    EMIN-J & 8 & $8 \cdot 10^{-3}$ & 1.062 & 1.360 & 189 & 26.2 & 13.8 & 40.0 & 8.1 \\
    EMIN-GS & 1 & $10^{0}$ & 1.061 & 1.346 & 276 & 23.0 & 19.9 & 42.9 & 5.3 \\
    EMIN-GS & 2 & $9 \cdot 10^{-2}$ & 1.062 & 1.352 & 221 & 26.2 & 16.0 & 42.2 & 8.2 \\
    EMIN-GS & 4 & $6 \cdot 10^{-3}$ & 1.062 & 1.363 & 184 & 31.7 & 14.3 & 46.0 & 13.8 \\
    EMIN-GS & 8 & $7 \cdot 10^{-4}$ & 1.063 & 1.367 & 183 & 42.9 & 13.8 & 56.7 & 24.8 \\
    \bottomrule
  \end{tabular}
\end{table}

Regarding algorithmic complexity, each energy minimization iteration with Gauss-Seidel
should cost exactly twice an iteration with Jacobi, and from the above
tests, Gauss-Seidel is able to reduce the energy more than twice as fast as Jacobi.
Thus Gauss-Seidel should be cheaper than Jacobi. Unfortunately, this is not
confirmed by our numerical experiments and this is likely due to a sub-optimal parallel implementation of the Gauss-Seidel preconditioner.
Although the block diagonal structure of $K$ allows theoretically for a perfectly
parallel implementation, the Gauss-Seidel implementation
requires more communication and synchronization stages than Jacobi, likely leading to our results where the Jacobi preconditioner is faster. A more cost effective
implementation of Gauss-Seidel will be the focus of future work.  For now, Jacobi is chosen
as our default preconditioner and will be used in all subsequent cases.

Finally, we observe that a relative energy reduction of one order of magnitude gives generally
the best trade-off between set-up time and AMG convergence. Therefore, we use $\tau = 0.1$
as default.

\subsection{Weak scalability test}

Here, we carry out a weak scalability study of energy minimization AMG for
linear elasticity on a unit cube and different levels of refinement.
The unit cube $[0,1]^3$ is discretized by regular tetrahedral elements, the material is
homogeneous, and all displacements are prevented on the square region
$[0,0.125] \times [0,0.125]$ at $z = 0$.
The mesh sizes are chosen such that each subsequent refinement produces about
twice the number degrees of freedom with respect to the previous mesh. The problem sizes
range from 222k rows to 124M rows, with an average entries-per-row ratio of 44.47.

Two sets of AMG parameters are used with Chronos: the first set targets a constant PCG iteration count, while
the second minimizes the total solution time ($T_t$). The first section of Table \ref{tab:resWeakCube} provides the
outcome of the first test. The iteration count increases very slowly from 23
to 33, while the problem size increases by a factor of about $2^9$. However, this nearly
optimal scaling comes with relatively large complexities.
As a consequence, total times are also relatively large, especially the setup time ($T_{p}$).
The time for energy minimization ($T_i$) scales quite well, however, with only a factor of 2 difference between the first
and last refinement levels.

Next to reduce the setup and solution times, we increase the AMG strength threshold \cite{RuStu1987} to allow for
lower complexities. The second section of Table \ref{tab:resWeakCube} collects the new outputs. It can be seen
that the relative trends are the same as in the previous tests, but that the timings are lower (30\% on average). The reduced
complexity is thus advantageous by providing faster wall-clock times, even at the expense of more iterations $n_{it}$.

\begin{table}
  \centering
  \small
  \caption{Weak scalability results for regular cube. Three sets of results are
    reported: i) energy minimization-based AMG via Chronos to produce an almost constant
    iteration count; ii) the same tuned to minimize the total time; and iii) PETSc's GAMG
    using all default parameters with the exception of $\mu$, which is chosen to reduce
    the overall solution time.}
  \label{tab:resWeakCube}
  \begin{tabular}{rr|rrrrrrrrrr}
    \toprule
    \multicolumn{2}{c|}{solver} & nrows & $n_{cr}$ & $N$ & $C_{gr}$ & $C_{op}$ & $n_{it}$ & $T_p$ [s] & $T_s$ [s] &
      $T_t$ [s] & $T_i$ [s] \\
    \midrule
    \parbox[t]{6pt}{\multirow{10}{*}{\rotatebox[origin=c]{90}{energy minimization}}}
    \parbox[t]{0pt}{\multirow{10}{*}{\rotatebox[origin=c]{90}{minimal iteration count}}}
    &&     222k &   1 &  1 & 1.077 & 1.540 & 23 & 28.7 & 2.7 & 31.4 & 0.12 \\
    &&     447k &   2 &  1 & 1.076 & 1.559 & 24 & 33.0 & 2.9 & 35.9 & 0.12 \\
    &&     902k &   4 &  1 & 1.076 & 1.580 & 26 & 36.1 & 3.4 & 39.5 & 0.13 \\
    &&   1,778k &   8 &  1 & 1.075 & 1.596 & 27 & 39.0 & 3.6 & 42.7 & 0.13 \\
    &&   3,675k &  16 &  1 & 1.075 & 1.610 & 28 & 43.3 & 4.1 & 47.4 & 0.15 \\
    &&   7,546k &  32 &  1 & 1.075 & 1.620 & 29 & 53.9 & 5.3 & 59.2 & 0.18 \\
    &&  15,533k &  64 &  2 & 1.075 & 1.630 & 30 & 61.2 & 5.9 & 67.2 & 0.20 \\
    &&  31,081k & 128 &  4 & 1.075 & 1.670 & 30 & 74.9 & 6.5 & 81.4 & 0.22 \\
    &&  62,391k & 256 &  8 & 1.075 & 1.642 & 32 & 72.8 & 6.8 & 79.7 & 0.21 \\
    && 124,265k & 512 & 16 & 1.075 & 1.646 & 33 & 88.8 & 7.8 & 96.7 & 0.24 \\
    \midrule
    \parbox[t]{6pt}{\multirow{10}{*}{\rotatebox[origin=c]{90}{energy minimization}}}
    \parbox[t]{0pt}{\multirow{10}{*}{\rotatebox[origin=c]{90}{best solution time}}}
    &&     222k &   1 &  1 & 1.047 & 1.252 & 31 & 19.5 &  3.1 & 22.6 & 0.10 \\
    &&     447k &   2 &  1 & 1.047 & 1.260 & 31 & 22.6 &  3.3 & 25.9 & 0.11 \\
    &&     902k &   4 &  1 & 1.047 & 1.269 & 34 & 24.4 &  3.8 & 28.2 & 0.11 \\
    &&   1,778k &   8 &  1 & 1.047 & 1.274 & 34 & 27.4 &  4.0 & 31.4 & 0.12 \\
    &&   3,675k &  16 &  1 & 1.047 & 1.279 & 37 & 30.6 &  4.7 & 35.3 & 0.13 \\
    &&   7,546k &  32 &  1 & 1.047 & 1.283 & 38 & 37.8 &  6.1 & 43.9 & 0.16 \\
    &&  15,533k &  64 &  2 & 1.046 & 1.286 & 49 & 44.3 &  8.5 & 52.8 & 0.17 \\
    &&  31,081k & 128 &  4 & 1.046 & 1.289 & 41 & 45.0 &  7.5 & 52.5 & 0.18 \\
    &&  62,391k & 256 &  8 & 1.046 & 1.291 & 43 & 48.4 &  8.6 & 57.0 & 0.20 \\
    && 124,265k & 512 & 16 & 1.046 & 1.292 & 51 & 52.5 & 10.8 & 63.3 & 0.21 \\
    \midrule
    \parbox[t]{6pt}{\multirow{10}{*}{\rotatebox[origin=c]{90}{GAMG ($\mu = 0.01$)}}}
    \parbox[t]{0pt}{\multirow{10}{*}{\rotatebox[origin=c]{90}{best solution time}}}
    &&     222k &   1 &  1 & N/A & 1.479 &  58 &  11.19 & 13.72 &  24.92 & 0.24 \\
    &&     447k &   2 &  1 & N/A & 1.503 &  69 &  10.02 & 17.12 &  27.14 & 0.25 \\
    &&     902k &   4 &  1 & N/A & 1.526 &  73 &  11.41 & 19.19 &  30.59 & 0.26 \\
    &&   1,778k &   8 &  1 & N/A & 1.549 &  78 &  12.63 & 21.05 &  33.68 & 0.27 \\
    &&   3,675k &  16 &  1 & N/A & 1.571 &  82 &  14.83 & 24.86 &  39.68 & 0.30 \\
    &&   7,546k &  32 &  1 & N/A & 1.608 &  86 &  23.04 & 35.52 &  58.56 & 0.41 \\
    &&  15,533k &  64 &  2 & N/A & 1.618 &  93 &  27.15 & 39.33 &  66.48 & 0.42 \\
    &&  31,081k & 128 &  4 & N/A & 1.703 &  97 &  37.48 & 41.58 &  79.06 & 0.43 \\
    &&  62,391k & 256 &  8 & N/A & 1.803 &  98 &  58.98 & 43.37 & 102.35 & 0.44 \\
    && 124,265k & 512 & 16 & N/A & 1.953 & 100 & 119.16 & 50.29 & 169.45 & 0.50 \\
    \bottomrule
  \end{tabular}
\end{table}

For comparison, the same set of problems is next solved with GAMG from PETSc. The
default values for all parameters are used, as it turned out they were already the best.
The only exception is the threshold for dropping edges in the aggregation graph ($\mu$),
whose value is reported alongside the results.
The third section of Table \ref{tab:resWeakCube} shows the output for GAMG.
The complexities and run-times are higher than for Chronos, especially
for the setup stage. The iteration counts are usually between 2 and 3 times larger than
those required by Chronos. For this test problem, reducing complexity by setting $\mu = 0.0$ is not beneficial as the increase in the iteration count cancels the set-up time reduction.
We also note that both the operator complexity and set-up
time increase significantly with the refinement level, while energy minimization
is able to better limit growth in these quantities. Comparing two codes is fraught with difficulty, but these results do indicate that energy minimization is an efficient approach in parallel for this problem.

\subsection{Challenging Real World Problems}

We now examine the proposed energy minimization approach for a set of challenging
real-world problems arising from discretized PDEs in both fluid dynamics and
mechanics. The former class consists of problems from the discretization of the
Laplace operator, such as underground fluid flow, compressible or incompressible airflow
around complex geometries, or porous flow. The latter class consists of mechanical
applications such as subsidence analysis, hydrocarbon recovery, gas storage
(geomechanics), mesoscale simulation of composite materials (mesoscale), mechanical
deformation of human tissues or organs subjected to medical interventions (biomedicine),
and design and analysis of mechanical elements, e.g., cutters, gears, air-coolers
(mechanical). The selected problems are not only large but also characterized by
severe ill-conditioning due to mesh distortions, material heterogeneity, and anisotropy.
They are listed in Table \ref{tab:matrices} with details
about the size, the number of nonzeros, and the application field from which they arise.

\begin{table}
  \centering
  \small
  \caption{Matrix sizes and number of non-zeroes for the real-world problems.}
  \label{tab:matrices}
  \begin{tabular}{lrrrr}
    \toprule
    matrix & nrows & nterms & avg nt/row & application \\
    \midrule
    \texttt{guenda11m} &  11,452,398 &    512,484,300 & 44.75 & geomechanics \\
    \texttt{agg14m}    &  14,106,408 &    633,142,730 & 44.88 & mesoscale \\
    \texttt{tripod20m} &  19,798,056 &    871,317,864 & 44.01 & mechanical \\
    \texttt{M20}       &  20,056,050 &  1,634,926,088 & 81.52 & mechanical \\
    \texttt{wing}      &  33,654,357 &  2,758,580,899 & 81.97 & mechanical \\
    \texttt{Pflow73m}  &  73,623,733 &  2,201,828,891 & 29.91 & reservoir \\
    \texttt{c4zz134m}  & 134,395,551 & 10,806,265,323 & 80.41 & biomedicine \\
    \bottomrule
  \end{tabular}
\end{table}

Table \ref{tab:resReal} reports the AMG
performance on these benchmarks when using the energy minimization procedure, 
classical prolongation smoothing (one step of weighed-Jacobi on the tentative prolongation), and GAMG, respectively.
The overall best time is highlighted in boldface.
As before, please note that for GAMG only the threshold value for dropping edges in the
aggregation graph ($\mu$) has been tuned. All the other parameters are used with their
default value, as they were already the best.

With respect to classical prolongation smoothing, the energy minimization procedure
is able to reduce the complexities, in particular $C_{op}$, the setup time $T_p$, and
also the iteration count (with the only exception being \texttt{Pflow73m}).  It is the reduced complexities that allow energy minimization to achieve the lower setup time.
The overall gain in total time is in the range $5$--$55\%$ for all test cases.

Energy minimization also compares favorably with GAMG.
GAMG provides a faster total time than energy minimization-based AMG on
two cases out of seven, \texttt{agg14m} and \texttt{M20}, and a similar total time on
\texttt{tripod20m}. The situation is reversed on the most challenging
examples, where GAMG is significantly slower and is unable to solve
\texttt{Pflow73m}. We briefly note that, unexpectedly, the ParMETIS partitioning
significantly harms GAMG effectiveness on \texttt{c4zz134m}. For this case, we
report GAMG performance on the matrix with its native ordering and mark this test with
a `*'. Typically, the GAMG set-up time is faster than energy minimization,
but energy minimization allows for a preconditioner of higher quality, which
significantly reduces the total time in the most difficult cases.
Fig. \ref{fig:realWorldRelTime} collects all these results, reporting the relative total
times. The setup and solve phases are denoted by different shading.

\begin{table}
  \centering
  \small
  \caption{Results for the real-world cases using: i) energy minimization AMG; ii)
    classical prolongation smoothing; and iii) PETSc's GAMG. Default PETSc GAMG parameters
are used. Only $\mu$, i.e., the threshold for dropping edges in aggregation graph, is
changed. `*' means that the matrix has not been partitioned with ParMETIS. Please note
that only the threshold value for dropping edges in the aggregation graph ($\mu$)
has been tuned. All the other parameters are used with their default value, as it turned
out they were already the best.}
  \label{tab:resReal}
  \begin{tabular}{rr|rrrrrrrrrr}
    \toprule
    \multicolumn{2}{c|}{solver} & case & $N$ & $\mu$ & $C_{gr}$ & $C_{op}$ & $n_{it}$ & $T_p$ [s] & $T_s$ [s] & $T_t$ [s] \\
    \midrule
    \parbox[t]{6pt}{\multirow{7}{*}{\rotatebox[origin=c]{90}{energy}}}
    \parbox[t]{0pt}{\multirow{7}{*}{\rotatebox[origin=c]{90}{minimization}}}
    && \texttt{guenda11m} & 1 & N/A & 1.041 & 1.325 &  987 & 314.0 & 399.0 & \textbf{713.0} \\
    && \texttt{agg14m}    & 1 & N/A & 1.042 & 1.322 &   23 &  66.7 &   7.2 &  74.0 \\
    && \texttt{tripod20m} & 2 & N/A & 1.049 & 1.302 &  104 &  40.3 &  22.2 &  \textbf{62.6} \\
    && \texttt{M20}       & 2 & N/A & 1.055 & 1.304 &  111 &  98.0 &  40.2 & 138.0 \\
    && \texttt{wing}      & 8 & N/A & 1.055 & 1.297 &  140 &  47.2 &  25.3 & \textbf{ 72.5} \\
    && \texttt{Pflow73m}  & 4 & N/A & 1.028 & 1.101 & 1169 & 225.0 & 424.0 & \textbf{649.0} \\
    && \texttt{c4zz134m}  & 8 & N/A & 1.029 & 1.122 &  154 &  72.7 &  48.8 & \textbf{122.0} \\
    \midrule
    \parbox[t]{6pt}{\multirow{7}{*}{\rotatebox[origin=c]{90}{smoothed}}}
    \parbox[t]{0pt}{\multirow{7}{*}{\rotatebox[origin=c]{90}{prolongation}}}
    && \texttt{guenda11m} & 1 & N/A & 1.041 & 1.378 & 1771 & 307.0 & 750.0 & 1060.0 \\
    && \texttt{agg14m}    & 1 & N/A & 1.042 & 1.371 &   48 &  62.5 &  15.5 &   78.1 \\
    && \texttt{tripod20m} & 2 & N/A & 1.048 & 1.336 &  212 &  34.6 &  47.5 &   82.2 \\
    && \texttt{M20}       & 2 & N/A & 1.054 & 1.733 &  154 & 167.0 &  76.6 &  244.0 \\
    && \texttt{wing}      & 8 & N/A & 1.055 & 1.697 &  301 &  93.5 &  71.6 &  165.1 \\
    && \texttt{Pflow73m}  & 4 & N/A & 1.058 & 1.371 &  841 & 441.3 & 394.5 &  836.0 \\
    && \texttt{c4zz134m}  & 8 & N/A & 1.028 & 1.199 &  277 &  79.2 &  98.9 &  178.0 \\
    \midrule
    \multicolumn{2}{c|}{\parbox[t]{6pt}{\multirow{7}{*}{\rotatebox[origin=c]{90}{GAMG}}}}
    & \texttt{guenda11m} & 1 & 0.00 & N/A & 1.524 & 2237 & 22.3 & 1553.9 & 1576.1 &  \\
    && \texttt{agg14m}    & 1 & 0.00 & N/A & 1.557 &   33 & 20.6 &   32.2 &   \textbf{52.7} &  \\
    && \texttt{tripod20m} & 2 & 0.01 & N/A & 1.679 &   48 & 32.2 &   32.2 &   64.4 &  \\
    && \texttt{M20}       & 2 & 0.01 & N/A & 1.203 &   60 & 36.7 &   62.4 &   \textbf{99.1} &  \\
    && \texttt{wing}      & 8 & 0.01 & N/A & 1.204 &  250 & 34.0 &  108.4 &  142.4 &  \\
    && \texttt{Pflow73m}  & 4 &  --- & N/A &   --- &  --- &  --- &    --- &    --- &  \\
    && \texttt{c4zz134m}* & 8 & 0.01 & N/A & 1.233 &  156 & 110.9 &  250.38 & 361.33 &  \\
    \bottomrule
  \end{tabular}
\end{table}

\begin{figure}
  \centering
  \includegraphics[width=0.95\linewidth]{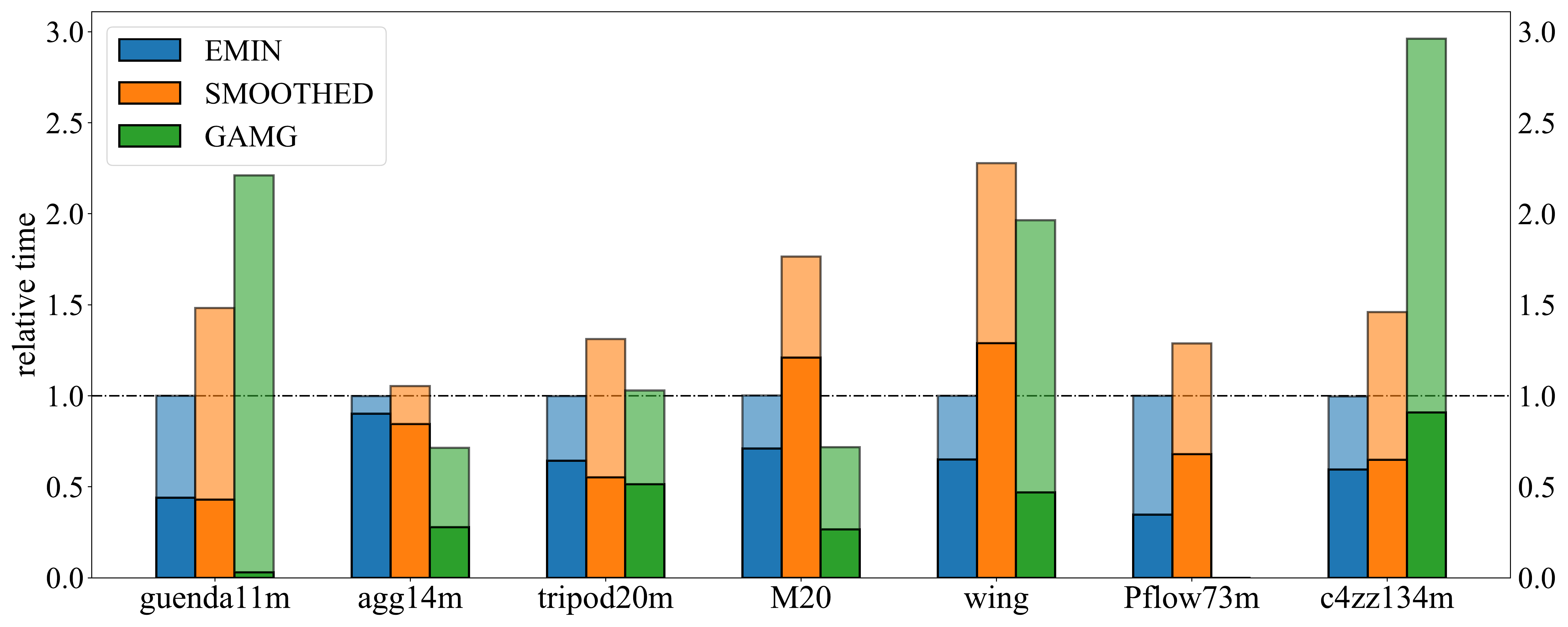}
  \caption{Comparison in terms of relative total time for the different preconditioners.
Darker portions of the bars represent AMG setup, while the lighter segments
represent the time spent in the solve phase. The relative baseline is time taken by the energy
minimization phase to build prolongation.}
  \label{fig:realWorldRelTime}
\end{figure}

\section{Conclusions}
\label{sec:concl}

This work provides evidence of the potential of a novel energy minimization procedure in
constructing the prolongation operator in AMG. While the theoretical advantages of this
idea are well known in the literature, the computational aspects and its practical
feasibility in a massively parallel software were still under discussion.

With this contribution, we have highlighted how the energy minimization approach can be
effectively implemented in a classical AMG setting, leading to robust and cost-effective 
prolongation operators when compared to other approaches. It is also shown how, especially
in challenging problems, this technique can lead to considerably faster AMG convergence.
The prolongation energy can be minimized with several schemes. We have adopted a restricted
conjugate gradient, accelerated by suitable preconditioners. We presented and analyzed
Jacobi and Gauss-Seidel preconditioners, restricting our attention to these two because
of their applicability in matrix-free mode. 
The experiments
show that Gauss-Seidel has the potential to be faster than Jacobi, however, its
efficient parallel implementation is not straightforward and needs more attention.

Weak scalability has been assessed on an elasticity
model problem by discretizing a cube with regular tetrahedra. The proposed algorithms have been
implemented in a hybrid MPI-OpenMP linear solver and its performance has been compared
to another well recognized AMG implementation (GAMG) using a set of large and difficult problems arising
from real-world applications.

In the future, we plan to further reduce set-up time by investigating other preconditioning
techniques as those described in~\cite{BenFac22}, and to extend this approach to non-symmetric
problems as well.


\bibliographystyle{siam}
\bibliography{mybibfile}

\begin{thebibliography}{10}

\bibitem{balay2019petsc}
{\sc S.~Balay, S.~Abhyankar, M.~Adams, J.~Brown, P.~Brune, K.~Buschelman,
  L.~Dalcin, A.~Dener, V.~Eijkhout, W.~Gropp, et~al.}, {\em Petsc users
  manual},  (2019).

\bibitem{BenGolLie99}
{\sc M.~Benzi, G.~H. Golub, and J.~Liesen}, {\em Numerical solution of saddle
  point problems}, Acta Numer., 14 (2005), pp.~1--137.

\bibitem{bootstrap2011}
{\sc A.~Brandt, J.~Brannick, K.~Kahl, and I.~Livshits}, {\em Bootstrap amg},
  SIAM Journal on Scientific Computing, 33 (2011), pp.~612--632.

\bibitem{BrBrKaLi2015}
{\sc A.~Brandt, J.~Brannick, K.~Kahl, and I.~Livshits}, {\em Algebraic distance
  for anisotropic diffusion problems: multilevel results}, Electronic
  Transactions on Numerical Analysis, 44 (2015), pp.~472--496.

\bibitem{BrMcRu1982}
{\sc A.~Brandt, S.~F. McCormick, and J.~W. Ruge}, {\em Algebraic multigrid
  ({AMG}) for automatic multigrid solution with application to geodetic
  computations}, tech. rep., Institute for Computational Studies, Colorado
  State University, 1982.

\bibitem{BrMcRu1984}
\leavevmode\vrule height 2pt depth -1.6pt width 23pt, {\em Algebraic multigrid
  ({AMG}) for sparse matrix equations}, in Sparsity and Its Applications, D.~J.
  Evans, ed., Cambridge Univ. Press, Cambridge, 1984, pp.~257--284.

\bibitem{brannick2018optimal}
{\sc J.~Brannick, F.~Cao, K.~Kahl, R.~D. Falgout, and X.~Hu}, {\em Optimal
  interpolation and compatible relaxation in classical algebraic multigrid},
  SIAM Journal on Scientific Computing, 40 (2018), pp.~A1473--A1493.

\bibitem{BraFal10}
{\sc J.~J. Brannick and R.~D. Falgout}, {\em {Compatible Relaxation and
  Coarsening in Algebraic Multigrid}}, SIAM Journal on Scientific Computing, 32
  (2010-01), pp.~1393 -- 1416.

\bibitem{BreFalMacManMccRug05}
{\sc M.~Brezina, R.~Falgout, S.~MacLachlan, T.~Manteuffel, S.~McCormick, and
  J.~Ruge}, {\em {Adaptive Smoothed Aggregation ({$\alpha$}SA) Multigrid}},
  SIAM Rev., 47 (2005), pp.~317--346.

\bibitem{BrHeMc2000}
{\sc W.~L. Briggs, V.~E. Henson, and S.~F. McCormick}, {\em A multigrid
  tutorial}, SIAM, Philadelphia, PA, USA, 2nd~ed., 2000.

\bibitem{de2006reducing}
{\sc H.~De~Sterck, U.~M. Yang, and J.~J. Heys}, {\em Reducing complexity in
  parallel algebraic multigrid preconditioners}, SIAM Journal on Matrix
  Analysis and Applications, 27 (2006), pp.~1019--1039.

\bibitem{FaVa2004}
{\sc R.~D. Falgout and P.~S. Vassilevski}, {\em On generalizing the algebraic
  multigrid framework}, SIAM J. Numer. Anal., 42 (2004), pp.~1669--1693.

\bibitem{CHRONOS-webpage}
{\sc M.~Frigo, G.~Isotton, and C.~Janna}, {\em {Chronos} {w}eb page}.
\newblock \url{https://www.m3eweb.it/chronos}, 2021.

\bibitem{Gor08}
{\sc S.~A. Goreinov, I.~V. Oseledets, D.~Savostyanov, E.~E. Tyrtyshnikov, and
  N.~L. Zamarashkin}, {\em {How to find a good submatrix}}, tech. rep., Nov.
  2008.

\bibitem{IsoFriSpiJan21}
{\sc G.~Isotton, M.~Frigo, N.~Spiezia, and C.~Janna}, {\em Chronos: a general
  purpose classical {AMG} solver for high performance computing}, SIAM J. Sci.
  Comput., 43 (2021), pp.~C335--C357.

\bibitem{Knu85}
{\sc D.~E. Knuth}, {\em {Semi-optimal bases for linear dependencies}}, Linear
  and Multilinear Algebra, 17 (1985), pp.~1--4.

\bibitem{BenFac22}
\leavevmode\vrule height 2pt depth -1.6pt width 23pt, {\em Solving linear
  systems of the form $(a + \gamma u u^t) x = b$ by preconditioned iterative
  methods}, Linear and Multilinear Algebra,  (2022).

\bibitem{parmetis}
{\sc K.~Lab}, {\em {ParMETIS - Parallel Graph Partitioning and Fill-reducing
  Matrix Ordering}}.
\newblock \url{http://glaros.dtc.umn.edu/gkhome/metis/parmetis/overview}, 2022.

\bibitem{maclachlan2006adaptive}
{\sc S.~MacLachlan, T.~Manteuffel, and S.~McCormick}, {\em Adaptive
  reduction-based amg}, Numerical Linear Algebra with Applications, 13 (2006),
  pp.~599--620.

\bibitem{paludetto2019novel}
{\sc V.~A.~P. Magri, A.~Franceschini, and C.~Janna}, {\em {A novel algebraic
  multigrid approach based on adaptive smoothing and prolongation for
  ill-conditioned systems}}, SIAM Journal on Scientific Computing, 41 (2019),
  pp.~A190--A219.

\bibitem{ManBreVan99}
{\sc J.~Mandel, M.~Brezina, and P.~Van\v{e}k}, {\em Energy optimization of
  algebraic multigrid bases}, Computing, 62 (1999), pp.~205--228.

\bibitem{manteuffel2019nonsymmetric}
{\sc T.~A. Manteuffel, S.~M\"{u}nzenmaier, J.~Ruge, and B.~Southworth}, {\em
  Nonsymmetric reduction-based algebraic multigrid}, SIAM Journal on Scientific
  Computing, 41 (2019), pp.~S242--S268.

\bibitem{ManOlsSchSou17}
{\sc T.~A. Manteuffel, L.~N. Olson, J.~B. Schroder, and B.~S. Southworth}, {\em
  A root-node-based algebraic multigrid method}, SIAM J. Sci. Comput., 39
  (2017), pp.~S723--S756.

\bibitem{manteuffel2018nonsymmetric}
{\sc T.~A. Manteuffel, J.~Ruge, and B.~S. Southworth}, {\em Nonsymmetric
  algebraic multigrid based on local approximate ideal restriction ($\ell$
  air)}, SIAM Journal on Scientific Computing, 40 (2018), pp.~A4105--A4130.

\bibitem{olson2010new}
{\sc L.~N. Olson, J.~Schroder, and R.~S. Tuminaro}, {\em A new perspective on
  strength measures in algebraic multigrid}, Numerical Linear Algebra with
  Applications, 17 (2010), pp.~713--733.

\bibitem{OlsSchTum11}
{\sc L.~N. Olson, J.~B. Schroder, and R.~S. Tuminaro}, {\em A general
  interpolation strategy for algebraic multigrid using energy minimization},
  SIAM J. Sci. Comput., 33 (2011), pp.~966--991.

\bibitem{RuStu1987}
{\sc J.~W. Ruge and K.~St{\"{u}}ben}, {\em Algebraic multigrid ({AMG})}, in
  Multigrid Methods, S.~F. McCormick, ed., Frontiers Appl. Math., SIAM,
  Philadelphia, 1987, pp.~73--130.

\bibitem{SalTum08}
{\sc M.~Sala and R.~S. Tuminaro}, {\em A new {P}etrov-{G}alerkin smoothed
  aggregation preconditioner for nonsymmetric linear systems}, SIAM J. Sci.
  Comput., 31 (2008), pp.~143--166.

\bibitem{TrOo2001}
{\sc U.~Trottenberg, C.~Oosterlee, and A.~Sch$\ddot{\mbox{u}}$ller}, {\em
  Multigrid}, Academic Press, London, UK, 2001.

\bibitem{PetscGuide}
{\sc L.~UChicago~Argonne and the PETSc Development~Team}, {\em {PCGAMG}}.
\newblock \url{https://petsc.org/main/docs/manualpages/PC/PCGAMG/index.html},
  2022.

\bibitem{Van96}
{\sc P.~Van{\v{e}}k, J.~Mandel, and M.~Brezina}, {\em {Algebraic multigrid by
  smoothed aggregation for second and fourth order elliptic problems}},
  Computing, 56 (1996), pp.~179--196.

\bibitem{vassilevski2008multilevel3}
{\sc P.~S. Vassilevski}, {\em {Multilevel block factorization preconditioners:
  Matrix-based analysis and algorithms for solving finite element equations}},
  Springer Science \& Business Media, 2008.

\bibitem{WaChSm2000}
{\sc W.~L. Wan, T.~F. Chan, and B.~Smith}, {\em An energy-minimizing
  interpolation for robust multigrid methods}, SIAM J. Sci. Comput., 21 (2000),
  pp.~1632--1649.

\bibitem{XuZik17}
{\sc J.~Xu and L.~Zikatanov}, {\em {Algebraic multigrid methods}}, Acta
  Numerica, 26 (2017), pp.~591--721.

\bibitem{xu2018ideal}
{\sc X.~Xu and C.-S. Zhang}, {\em On the ideal interpolation operator in
  algebraic multigrid methods}, SIAM Journal on Numerical Analysis, 56 (2018),
  pp.~1693--1710.

\end{thebibliography}

\end{document}